\documentclass[12pt,reqno]{amsart}
\usepackage{amsmath}
\usepackage{}
\usepackage{}
\usepackage{amsfonts}
\usepackage{a4wide}
\usepackage{amsmath,amsthm,amssymb,amscd}
\usepackage{latexsym}
\usepackage{hyperref}
\usepackage[numbers,sort&compress]{natbib}
\usepackage{hypernat}
\allowdisplaybreaks
\numberwithin{equation}{section}

\newtheorem{theorem}{Theorem}[section]
\newtheorem{proposition}[theorem]{Proposition}

\newtheorem{lemma}[theorem]{Lemma}

\theoremstyle{definition}

\newtheorem{definition}[theorem]{Definition}

\newtheorem{remark}[theorem]{Remark}

\def\R{\mathbb R}
\def\N{\mathbb N}

\begin{document}

\title[Weakly coupled Schr\"{o}dinger system]{Nonstandard solutions for a perturbed nonlinear Schr\"{o}dinger system with small coupling coefficients\protect\thanks{A perturbed nonlinear Schr\"{o}dinger system.}}

\author{Xiaoming An}

 \address{School of Mathematics and Statistics\, \&\, Guizhou University of Financeand Economics, Guiyang, 550025, P. R. China}
\email{651547603@qq.com}

\author{Chunhua Wang$^*$}

\address{School of Mathematics and Statistics\, \&\, Hubei Key Laboratory of Mathematical Sciences, Central China Normal
University, Wuhan, 430079, P. R. China}

\email{chunhuawang@mail.ccnu.edu.cn}
%

%
\thanks{$^*$ Corresponding author: Chunhua Wang}

\begin{abstract}
{In this paper, we consider the following weakly coupled nonlinear Schr\"{o}dinger system
\begin{equation*}
  \left\{
    \begin{array}{ll}
       -\epsilon^{2}\Delta u_1 + V_1(x)u_1 = |u_1|^{2p - 2}u_1 + \beta|u_1|^{p - 2}|u_2|^pu_1, & x\in \mathbb{R}^N,\vspace{0.12cm}\\
       -\epsilon^{2}\Delta u_2 + V_2(x)u_2 = |u_2|^{2p - 2}u_2 + \beta|u_2|^{p - 2}|u_1|^pu_2, & x\in \mathbb{R}^N,
    \end{array}
  \right.
\end{equation*}
where $\epsilon>0$, $\beta\in\mathbb{R}$ is a coupling constant, $2p\in (2,2^*)$ with $2^* = \frac{2N}{N - 2}$ if $N\geq 3$ and $+\infty$ if $N = 1,2$, $V_1$ and $V_2$ belong to $C(\mathbb{R}^N,[0,\infty))$.

When $p\ge 2$ and $\beta>0$ is suitably small,
we show that the problem has a family of nonstandard solutions $\{w_{\epsilon} = (u^1_{\epsilon},u^2_{\epsilon}):0<\epsilon<\epsilon_{0}\}$ concentrating synchronously at the common local minimum of $V_1$ and $V_2$.
 All decay rates of $V_i(i=1,2)$ are admissible
 and we can allow that $\beta>0$ is close to $0$ in this paper.
Moreover, the location of concentration points is given by local Pohozaev identities. Our proofs are based on variational methods and the
penalized technique.}

\end{abstract}

\maketitle
\section{Introduction}\label{s1}
In this paper, we study the following nonlinear Schr\"odinger system
\begin{align}\label{eq1.1}
  \left\{
    \begin{array}{ll}
       -\epsilon^{2}\Delta u_1 + V_1(x)u_1 = |u_1|^{2p - 2}u_1 + \beta|u_1|^{p - 2}|u_2|^pu_1, & x\in \mathbb{R}^N,\vspace{0.12cm}\\
       -\epsilon^{2}\Delta u_2 + V_2(x)u_2 = |u_2|^{2p - 2}u_2 + \beta|u_2|^{p - 2}|u_1|^pu_2, & x\in \mathbb{R}^N,
    \end{array}
  \right.
\end{align}
where $N \geq 1$, $V_i\in C(\mathbb{R}^N,[0,\infty)), i =1,2$,
$\beta \in \R$ is a coupling constant. This type of systems arise when one considers standing waves of time-dependent $k$-coupled Schr\"odinger systems with $k = 2$ of the form
\begin{equation}\label{eq1.2}
  \left\{
    \begin{array}{ll}
      i\epsilon\frac{\partial\psi_j}{\partial t} = \epsilon^2\Delta\psi_j - U_j(x)\psi_j + \alpha_j|\psi_j|^{2p - 2}\psi_j + |\psi_j|^{p - 2}\psi_j\sum_{l = 1,s\neq j}^k\beta_{js}|\psi_{l}|^p
       , & \text{in}\ \mathbb{R}^N, \vspace{0.12cm}\\
      \psi_j = \psi_j(x,t)\in\mathbb{C},\ t > 0,\ j = 1,\ldots,k, &
    \end{array}
  \right.
\end{equation}
where $\epsilon > 0$, $i$ denotes the imaginary part, $\alpha_j$ and $\beta_{js} = \beta_{sj}$ are coupling constants.

In Physics, system \eqref{eq1.2} is applied to study the nonlinear optics in isotropic materials, for instance the propagation pulses in fiber. Because of the appearance of birefringence, a pulse $\psi$ tends to be spilt into two pulses ($\psi_1,\psi_2$) in the two polarization directions, but Menyuk \cite{M} proved that the two components $\psi_1,\psi_2$ in a birefringence optical fiber are governed by the two coupled nonlinear Schr\"odinger system in \eqref{eq1.2}.

System $\eqref{eq1.2}$  also has applications in Bose-Einstein condensates theory. For example, when $k = 2$ in \eqref{eq1.2}, $\psi_1$ and $\psi_2$ are the wave functions of the corresponding condensates and $\beta$ is the interspecies scattering length. Physically,
  $\beta>0$ is known as
 the attractive case, and the components of states tend to get along with each other leading to synchronization. Whereas $\beta<0$ is the repulsive case,
  the components tend to segregate each other, leading to phase separations.

In recent years, a lot of works such as the existence of ground states, least energy solitary waves, infinitely many segregated and synchronized solutions and so on
have been done for \eqref{eq1.1} in the case that $\epsilon > 0$ is fixed, see \cite{WW,S3,S1,SW,M2,PW,LW3,LW1,LLW,B5,BW,B1,AE,MMP,M} and their references therein. Hereafter we say a vector function $w = (u_1,u_2)$ is nonstandard if $u_1,u_2\neq0$.

%
%
%
%
%
%
%
%

In the last three decades, a large amount of semiclassical analysis has been done on problem \eqref{eq1.1} with $\beta = 0$, i.e.,
\begin{equation*}\label{P}
-\epsilon^{2}\Delta u + V(x) u = |u|^{2p - 2}u,\,\,\,x\in \mathbb{R}^N,
\end{equation*}
where $V$ is an external potential. We refer the readers to \cite{ABC-97-ARMA, AM-06, AMS-01-ARMA,9,FW-86-JFA,O-93-CMP,WZ-97-SIAM,BL-83-ARMA,C-72-ARMA,S-77-CMP} and the references therein, where
under various hypotheses on the potential $V(x)$, solutions which exhibit a spike shape around the critical points of $V(x)$ were obtained.

Hereafter, we say a solution $w_{\epsilon}=(u^1_{\epsilon},u^2_{\epsilon})$ of \eqref{eq1.1} is nonstandard if $\liminf\limits_{\epsilon\to0}\|u^i_{\epsilon}\|_{L^{\infty(\R^N)}}>0$($i=1,2$). To our best knowledge,  the only result about the existence of nonstandard solutions for \eqref{eq1.1} is
proved by Montefusco, Pellacci, Squassina
 in \cite{MPS}, where they showed
 that \eqref{eq1.1} has a family of nonstandard solutions concentrating around the common local minimum of $V_1$ and $V_2$ provided that $\beta> 0$ is suitably large, $\inf_{\mathbb{R}^N}V_i > c_i > 0$$(i = 1,2)$ and $p = 2$.
  Hence we intend to obtain nonstandard solutions for \eqref{eq1.1} when $\beta> 0$ is close to $0$.
   The potential we consider can decay faster than $|x|^{-2}$ or be compactly supported,
   in which case the existence of nonstandard solutions for  \eqref{eq1.1} is analogue to the conjecture
    proposed by Ambrosetti and Malchiodi in \cite{AM-061}
     about whether there exist solutions for \eqref{P} when $\lim_{|x|\to\infty}V(x)|x|^2=0$. In the single case $\beta = 0$, this conjecture was answered partially in \cite{BaNa,Byeon-Wang,Cao-Peng} and positively in \cite{MV}.

One main method used in this paper is the penalized skill. When $\beta = 0$, this method was introduced firstly in \cite{CP,9} if $\inf_{\mathbb{R}^N}V>0$ and developed in \cite{28,30,16} if $V$ vanishes. Basing on the penalized idea, we create a new penalized function to cut off the nonlinear term in \eqref{eq1.1}. The other main method is mathematical analysis, which is used to do some monotonicity and zero point analysis(see Section \ref{s2}).

We set the Hilbert space $\mathcal{H}$ as
$$
\mathcal{H} = \{w = (u_1,u_2): u_1,u_2\in H^1(\mathbb{R}^N)\},
$$
with inner product
\begin{equation*}
  \langle w^1, w^2\rangle_{\mathcal{H}}
= \int_{\mathbb{R}^N}(\nabla u_{11}\nabla u_{21} + \lambda_1u_{11}u_{21})
 + \int_{\mathbb{R}^N}(\nabla u_{12}\nabla u_{22} + \lambda_2u_{12}u_{22})
\end{equation*}
and its reduced norm
\begin{equation*}
  \|w\|^2_{\lambda_1,\lambda_2} = \|u_{11}\|^2_{\lambda_1} + \|u_{12}\|_{\lambda_2}^2
\end{equation*}
for all $w^i = (u_{i1},u_{i2})(i = 1,2) \in\mathcal{H}$, where $\lambda_i\in(0,+\infty),$  $\|\cdot\|_{\lambda_i} = \int_{\mathbb{R}^N}|\nabla \cdot|^2 + \lambda_i|\cdot|^2$($i=1,2$) stand for the equivalent norms in $H^1(\mathbb{R}^N)$.

According to different decay rates of $V$, we define for each $i=1,2$ the weighted Hilbert space $H^1_{V_i,\epsilon}$ as

$$
H^1_{V_i,\epsilon}=\left\{
  \begin{array}{ll}
    \big\{\epsilon\nabla u\in L^2(\mathbb{R}^N):\,\,u\in L^2(\mathbb{R}^N,V_i(x)dx)\big\},
    & \text{if}\ \liminf_{|x|\to\infty}V_i(x)|x|^2>0, \vspace{0.12cm} \\
    \big\{u\in D^{1,2}(\R^N):\,\,u\in L^2(\mathbb{R}^N,V_i(x)dx)\big\}, & \text{others}
  \end{array}
\right.
$$
endowed with the norm
$$
\|u_i\|^2_{V_i,\epsilon}(\mathbb{R}^N) = \int_{\mathbb{R}^N}(\epsilon^{2}|\nabla u_i|^2 + V_i(x)u_i^2)dx,\ \,\, \forall u_i\in H^1_{V_i,\epsilon}(\mathbb{R}^N),
$$
where $D^{1,2}(\R^N)$ is the completion of $C^{\infty}_c(\R^N)$ with respect to the norm $\|u\|^2=\int_{\R^N}|\nabla u|^2dx$.
Also, like $\mathcal{H}$, we define the weighted product Hilbert space $\mathcal{H}_{\epsilon}(\mathbb{R}^N)$ as
$$
\mathcal{H}_{\epsilon}
=\{w=(u_1,u_2):\,\, u_i\in H^1_{V_i,\epsilon}(\mathbb{R}^N),\ i=1,2\},
$$
with inner product
\begin{equation*}
  \langle w^1, w^2\rangle_{\mathcal{H}_{\epsilon}}
= \sum_{i = 1}^2\int_{\mathbb{R}^N}\big(\epsilon^2\nabla u_{1i}\nabla u_{2i} + V_{i}(x) u_{1i}u_{2i}\big)
\end{equation*}
and its reduced norm
\begin{equation*}
  \|w\|^2_{\epsilon} = \sum_{i = 1}^2\|u_i\|^2_{V_i,\epsilon}
\end{equation*}
for all $w^i = (u_{i1},u_{i2}),\ w= (u_1,u_2)\in \mathcal{H}_{\epsilon}$.

In the sequel, we set
$$
2^* := \frac{2N}{(N - 2)_+} = \left\{
          \begin{array}{ll}
            \frac{2N}{N - 2}, & \text{for}\ N \geq 3, \vspace{2mm}\\
             +\infty, & \text{for}\ N = 1, 2.
          \end{array}
        \right.
$$

We assume that for every $i = 1,2$, $V_i\in C(\mathbb{R}^N,[0,\infty))$ satisfies the following assumptions. There exist open bounded sets $\Lambda,U$

\begin{equation}\label{AAeq1.4}
0 < \textsf{m}_i = \inf_{\Lambda}V_i \leq \inf_{U\backslash\Lambda}V_i(x),\,\, \Lambda\subset\subset U,
\end{equation}
and
\begin{equation}\label{AAeq1.6}
  \inf_{\Lambda}(V_1(x) + V_2(x))< \inf_{U\backslash\Lambda}(V_1(x) + V_2(x)).
\end{equation}
Denoting
\begin{equation*}\label{AAeq1.5}
V_{\min}(x)= \min\{V_1(x),V_2(x)\},\ \forall x\in\mathbb{R}^N,\, \mathcal{M} = \cap_{i = 1,2}\{x\in\Lambda:V_i(x) = \textsf{m}_i\},
\end{equation*}
we assume that
\begin{equation}\label{eq1.12}
  \mathcal{M}\ne \emptyset.
\end{equation}

Without loss of generality, we assume that $0\in\Lambda$,  $\textsf{m}_1\leq \textsf{m}_2$ and denote
$$
\omega = \textsf{m}_1/\textsf{m}_2\in (0,1],
$$
which represents the ratio of two pulses in Physics.

{Hereafter, we say a solution $w_{\epsilon} = (u^1_{\epsilon},u^2_{\epsilon})$ of \eqref{eq1.1} is nonstandard if $\liminf_{\epsilon\to 0}\|u^i_{\epsilon}\|_{L^{\infty}(\mathbb{R}^N)}>0$ for all $i = 1,2$ and standard if $\lim_{\epsilon\to 0}\|u^i_{\epsilon}\|_{L^{\infty}(\mathbb{R}^N)}=0$ for some $i \in\{1,2\}$.}

When $\beta>0$ is close to $0$, in order to construct solutions for \eqref{eq1.1} we need the following theorem. Define for every $\alpha_1,\alpha_2>0,\beta\in\mathbb{R}$
{\begin{equation}\label{eq1.6}
  \left\{
    \begin{array}{ll}
      \mathcal{C}_{\alpha_1,\beta_+}:=J_{\alpha_1,\beta_+}(U_{\alpha_1,\beta_+}), &  \\
      U_{\alpha_1,\beta_+}\ \text{is the ground states of}\ -\Delta u + \alpha_1 u = (1 + \beta_+)|u|^{2p - 2}u, &\vspace{0.12cm}\\
      J_{\alpha_1,\beta_+}(w) := \frac{\|u\|^2_{\alpha_1}}{2} - \frac{1 + \beta_+}{2p}|u|^{2p}_{2p},\ \ \forall u\in H^1(\mathbb{R}^N),&
    \end{array}
  \right.
\end{equation}
and
\begin{equation}\label{eq1.7}
  \left\{
    \begin{array}{ll}
      \mathcal{C}^*_{\alpha_1,\alpha_2,\beta}:=\max_{t,s>0}J_{\alpha_1,\alpha_2,\beta}\big((tU_{\alpha_1,\beta_+},sU_{\alpha_2,\beta_+})\big), &  \\
      J_{\alpha_1,\alpha_2,\beta}(u) := \|w\|^2_{\alpha_1,\alpha_2}/2 - \frac{1}{2p}|u_1|^{2p}_{2p} - \frac{1}{2p}|u_2|^{2p}_{2p} - \frac{\beta}{p}|u_1u_2|^p_p,&\\
          \qquad\qquad\qquad \forall w = (u_1,u_2)\in \mathcal{H}(\mathbb{R}^N).&
    \end{array}
  \right.
\end{equation}}
We have:
\begin{theorem}\label{important}
For every $2p\in (2,2_*)$, there exist
a constant $\tilde{\beta}_{\textsf{m}_1,\textsf{m}_2,p}>0$ and a decreasing function $\vartheta(\cdot):(0,\tilde{\beta}_{\textsf{m}_1,\textsf{m}_2,p})\to (0,+\infty)$ such that if
\begin{equation}\label{neweq2}
 0<\beta <\tilde{\beta}_{\textsf{m}_1,\textsf{m}_2,p},
\end{equation}
then it holds
\begin{eqnarray}\label{eq1.4}
\begin{split}
\left\{
  \begin{array}{ll}
    \text{for every}\ k\in\mathbb{N},\ \text{there exists no}\
 \delta_1,\ldots,\delta_k\in [0,\vartheta(\beta))&\vspace{0.12cm}\\
\text{such that}\
 \sum_{i = 1}^k\mathcal{C}_{\textsf{m}_1 + \delta_i,0}\in [\mathcal{C}_{\textsf{m}_1,\beta} + \mathcal{C}_{\textsf{m}_2,\beta}, \mathcal{C}^*_{\textsf{m}_1,\textsf{m}_2,\beta}].&
  \end{array}
\right.
\end{split}
\end{eqnarray}
\end{theorem}
The proof of Theorem \ref{important} will be given at the end of Section \ref{s2}.

\begin{theorem}\label{th1.2}
Let $N \in\mathbb{N}$, $1\le N\le 3$, $2p \in [4,2^*)$ and $V_i\in C(\mathbb{R}^N;[0,\infty))$(i=1,2). Assume that either
\begin{equation}\label{AAeq1.8}
       \liminf_{|x|\to\infty}V_{min}(x)|x|^{2\sigma} > 0,\ \ \sigma\in[0,1],
\end{equation}
or
\begin{equation}\label{AAeq1.9}
  N\geq 3\ \ \text{and}\ \ 2p - 2 > \frac{2}{N - 2}.
\end{equation}
Then, letting $\tilde{\beta}_{\textsf{m}_1,\textsf{m}_2,p}$ be the constant in Theorem \ref{important}, there exists a constant $\bar{\beta}\in(0, \tilde{\beta}_{\textsf{m}_1,\textsf{m}_2,p}]$ such that problem \eqref{eq1.1} has a family of nonstandard solutions $\{w_{\omega,\epsilon} = ({u}^1_{\omega,\epsilon},{u}^2_{\omega,\epsilon}):0<\epsilon<\epsilon_{0}\}\in\mathcal{H}_{\omega,\epsilon}$ if $0<\beta<\bar{\beta}$ and
\begin{equation}\label{reviseeq}
\sup_{x\in\Lambda}V_1(x) - \textbf{m}_1 <\vartheta(\beta),\ \sup_{x\in\Lambda}\mathcal{C}_{V_2(x),0} <\mathcal{C}_{\textbf{m}_1,\beta} + \mathcal{C}_{\textbf{m}_2,\beta},
\end{equation}
where $\vartheta(\beta)$ is the constant given by Theorem \ref{important}.
 Moreover, supposing that ${u}^i_{\omega,\epsilon}({x}^i_{\omega,\epsilon}) =\max\limits_{x\in\mathbb{R}^N} {u}^i_{\omega,\epsilon}(x)$($i = 1,2$) and $\big({u}^1_{\omega,\epsilon} + {u}^2_{\omega,\epsilon}\big)({x}_{\omega,\epsilon}) = \max_{x\in\mathbb{R}^N}\big({u}^1_{\omega,\epsilon} + {u}^2_{\omega,\epsilon}\big)(x)$, then there hold
\begin{align*}
  &(i)\,\,{\liminf\limits_{\epsilon\to 0}u^i_{\omega,\epsilon}(x^i_{\omega,\epsilon})},\ \liminf\limits_{\epsilon\to 0}(u^1_{\omega,\epsilon} + u^2_{\omega,\epsilon})(x_{\omega,\epsilon})>0;\\
  &(ii)\,\,{\liminf_{\epsilon\to 0}\|u^i_{\omega,\epsilon}\|_{L^{\infty}(B_{\epsilon \rho}(x^i_{\omega,\epsilon}))}} > 0,\ {\liminf_{\epsilon\to 0}\|u^1_{\omega,\epsilon} + u^2_{\omega,\epsilon}\|_{L^{\infty}(B_{\epsilon \rho}(x_{\omega,\epsilon}))}} > 0;\\
  &(iii)\,\,\limsup_{\epsilon\to 0}\frac{|x^1_{\omega,\epsilon} - x^2_{\omega,\epsilon}| + |x^1_{\omega,\epsilon} - x_{\omega,\epsilon}|}{\epsilon} < +\infty;\\
  &(iv)\,\,\lim_{\epsilon\to 0}dist\ ({x}_{\omega,\epsilon},\mathcal{M}) = 0;\\
  &(v)\,\,\big(u^1_{\omega,\epsilon} + u^2_{\omega,\epsilon})(x)\le\left\{
             \begin{array}{ll}
              h(\sigma)Ce^{-c_{\epsilon,\sigma}\frac{|x - x_{\omega,\epsilon}|^{1 - \sigma}}{\epsilon^{1 - \sigma}}} + \frac{(1 - h(\sigma))C\epsilon^{C_{N,\epsilon}}}{1 + |x|^{C_{N,\epsilon}}}, & \text{if \eqref{AAeq1.8}}\ \text{holds}\vspace{0.12cm}\\
 Ce^{-\frac{\textsf{m}_1|x - x_{\omega,\epsilon}|}{\epsilon(|1 + |x - x_{\omega,\epsilon}|)}}\frac{1}{1 + |x|^{N - 2}},& \text{if \eqref{AAeq1.9}}\ \text{holds},
             \end{array}
           \right.
\end{align*}
where $h(\sigma) = 1$ if $0\le\sigma <1$ and $h(\sigma)=0$ if $\sigma = 1$, $\lim_{\epsilon\to 0}c_{\epsilon,\sigma}, \lim_{\epsilon\to 0}C_{N,\epsilon} = +\infty$, and $C$ is a positive constant.
\end{theorem}

\begin{remark}
When $\omega=1$, the constant $\bar{\beta} =2^{p-1}-1$ and when $\omega<1$, it is given by \eqref{neweq12} in Section \ref{s4}. \eqref{reviseeq} and the conclusion \eqref{eq1.4} in Theorem \ref{important} are
 used to construct nonstandard solutions. Indeed, it is used to get a contradiction if the solution in Lemma \ref{le3.8} is standard, see Lemmas \ref{le3.9} and \ref{Cle3.8} for details. \eqref{eq1.12} is used to prove property $(iii)$, which is proved by Local Pohozaev identities when $\omega<1$ and by comparing energy when $\omega=1$. This property is necessary in constructing a penalized function,  see \eqref{space} in Section \ref{s5} below.
%
%

By the monotonicity of $\mathcal{C}_{\cdot,\cdot}$(see \eqref{AAAeq1.21}),
 we need $\sup_{x\in\Lambda}V_1(x) - \textbf{m}_1$ and $\sup_{x\in\Lambda}V_2(x)$
  being suitably small. This requirement is natural because we can rearrangement $\Lambda$ as a small neighbourhood of $\mathcal{M}$. Moreover, if we let $$\sup_{x\in\Lambda}\mathcal{C}_{V_2(x),0} \le \mathcal{C}_{\textbf{m}_1,\tilde{\beta}_{\textbf{m}_1,\textbf{m}_2,p}/2} + \mathcal{C}_{\textbf{m}_2,\tilde{\beta}_{\textbf{m}_1,\textbf{m}_2,p}/2}
$$
and
\begin{equation}\label{addeq1.1}
\sup_{\Lambda}V_1(x) - \textsf{m}_1\le \vartheta\big(\frac{\tilde{\beta}_{\textsf{m}_1,\textsf{m}_2,p}}{2}\big),
\end{equation}
then by the decreasing of $\vartheta(\beta)$, the size of $\Lambda$ can be fixed. Indeed, if $\Lambda$  satisfies \eqref{addeq1.1}, then
the conclusion in Theorem \ref{important} holds for all $\beta\in\big(0,\frac{\tilde{\beta}_{\textsf{m}_1,\textsf{m}_2,p}}{2}\big)$.

\end{remark}


The main difficulties in the proof of Theorem \ref{th1.2} lie in the following
two aspects. Firstly, we can not construct solutions by constructing the ground solution as usual,
since every nonstandard solution of the limit system of \eqref{eq1.1} must own higher energy when $\beta>0$ small and $p\ge 2$(see Section \ref{s2} for more details). To this end, we construct
a nonstandard solution $w_{\epsilon}=(u^1_{\epsilon},u^2_{\epsilon})$ via minimizing-maxmizing on a two dimensional mountain pass geometry(see Definition \ref{de3.5} below). The second difficulty is the proof of the concentration phenomena of $w_{\epsilon}$. Actually, there
is not any monotonicity formula for higher energy of a coupled system(even for a single equation), which makes us
 not get the concentration phenomena via comparing energy as before. To this end, we firstly establish an accurate lower bound for the energy of all nonstandard solutions of
the limit system corresponding to \eqref{eq1.1} with $\beta<1$ and $p\ge 2$, see Theorem \ref{th2.2} below(for its application, see \eqref{eq4.33}, \eqref{eq4.41} and \eqref{eq4.42} in Section \ref{s4} for more details). Then, by some special construction of mountain pass geometry in Definition \ref{de3.5}, we get lower and upper bounds about the mountain pass value ${\mathcal{C}}_{\epsilon}$, see \eqref{eq3.9} below. By the estimates of ${\mathcal{C}}_{\epsilon}$, the lower bounds in Theorem \ref{th2.2} and Theorem \ref{important}, we then get the concentration phenomena in Theorem \ref{th1.2}. See Sections \ref{s3} and \ref{s4} for more details.

The example in Section 3 says that the usual functional corresponding to system \eqref{eq1.1} is not well defined
if $V$ vanishes. Hence we have to cut off the nonlinear term in \eqref{eq1.1} by a function {named the penalized function as usual}. The penalized function $P_{\epsilon}$ is constructed to be dominated by the potential $V_{min}$ if $\liminf_{|x|\to+\infty}V_{min}(x)|x|^{2\sigma} > 0$ with $\sigma\in[0,1]$. But when $V_{min}$ vanishes faster than $|x|^{-2}$ or even has compact support, the penalized function may not be dominated by $V_{min}$ anymore. To this end, we construct the penalized function to be dominated by the Hardy potential $\frac{1}{|x|^2}$. Indeed, the construction bases on the positivity of Hardy's operator $-\Delta - \kappa |x|^{-2}$ with $\kappa<(N - 2)^2/4$. The penalized function here is new, which simplifies a lot of computation in this paper. We will make this more precise in Sections \ref{s3} and \ref{s5}.

As an interesting part, when $\omega<1$, we prove the property $(iii)$ Theorem \ref{th1.2} by using local Pohozaev identities, which were used recently to get the location of concentration points, see \cite{GPS,PWY} and the references therein. To apply these identities, we need the solution to decrease faster outside a small ball of ${x}_{\omega,\epsilon}$, see Lemma \ref{LE4.6} below. These decay estimates follow by the further construction of super solutions in Section \ref{s6}, which depends on our special construction of the penalized function in Section \ref{s5}.

Finally, we want to emphasize the case that $V_{min}$ vanishes faster than $|x|^{-2}$ and even has compact support is a conjecture proposed by Ambrosetti and Malchiodi in \cite{AA}. We construct intuitive penalized functions with respect to different vanishing
 rates of $V_{min}$, see \eqref{eq3.1} and Section \ref{s5} below. Our construction is suitable for all subcritical $2p$ and dimension $N$ when $\liminf_{|x|\to+\infty}V_{min}|x|^{2\sigma}>0$ with $\sigma\in[0,1]$, which is
also of great interest.

\vspace{0.2cm}
We organize this paper as follows. In section \ref{s2}, we give some key results about the limit system. In section \ref{s3}, we establish the penalized scheme. In section \ref{s4}, we study the concentration phenomena
for the penalized solution. {We use local Pohozaev identities to show the concentration points are in $\mathcal{M}$ when $\beta>0$ is small}. In Section \ref{s5}, we prove the penalized solutions obtained in Section \ref{s3} solve the original problem \eqref{eq1.1} by constructing a precise penalized function $P_{\epsilon}$. In section \ref{s6}, we use the penalized function in Section \ref{s5} to construct suitable super-solutions to get the decay estimates in Theorem \ref{th1.2}, which are needed during the use of Pohozaev identities.

\section{The Limit problem}\label{s2}
{In} this section, we give some results about the limit problem corresponding to \eqref{eq1.1}.
We give the proof of Theorem \ref{important} in this section.

The limit system corresponding to \eqref{eq1.1} is given by
\begin{equation}\label{eq2.1}
  \left\{
    \begin{array}{ll}
      -\Delta u_1 + \alpha_1 u_1 = |u_1|^{2p - 2}u_{1} + \beta|u_1|^{p - 2}u_1|u_2|^p, & \text{in}\ \mathbb{R}^N, \vspace{2mm}\\
      -\Delta u_2 + \alpha_2 u_2 = |u_2|^{2p - 2}u_{2} + \beta|u_2|^{p - 2}u_2|u_1|^p, & \text{in}\ \mathbb{R}^N,
    \end{array}
  \right.
\end{equation}
where $\alpha_1,\alpha_2>0,\beta\in\mathbb{R}$.
Its Euler-Lagrange fucntional $J_{\alpha_1,\alpha_2,\beta}$ was given in \eqref{eq1.7}.
One can use the same argument as that of  \cite{MW} to show that the ground energy
\begin{eqnarray}\label{eq2.2}
\begin{split}
 \mathcal{C}_{\alpha_1,\alpha_2,\beta}
&= \inf_{\gamma\in\Gamma_{\alpha_1,\alpha_2,\beta}}\max_{t\in[0,1]}
J_{\alpha_1,\alpha_2,\beta}(\gamma(t))=\inf_{w\geq 0,w\in C^{\infty}_c(\mathbb{R}^N)\backslash\{0\}}\max_{t>0}
J_{\alpha_1,\alpha_2,\beta}(tw)
\end{split}
\end{eqnarray}
can be achieved by a positive radial function $W_{\alpha_1,\alpha_2,\beta}$,
where
$$
\Gamma_{\alpha_1,\alpha_2,\beta}: = \big\{\gamma\in C([0,1],\mathcal{H}):\gamma(0) = 0\,\,\text{and}\,\,J_{\alpha_1,\alpha_2,\beta}(\gamma(1))<0\big\}.
$$

When $p\geq 2$, the ground state $W_{\alpha_1,\alpha_2,\beta}$ is $(U_{\alpha_1,0},0)$ if $\alpha_1\le \alpha_2$ and $(0,U_{\alpha_2,0})$ if $\alpha_2\le \alpha_1$ if $0<\beta\le2^{p - 1} - 1$(see \cite{MMP}), which implies $\beta> 2^{p - 1} - 1$ is necessary for $W_{\alpha_1,\alpha_2,\beta}$ being nonstandard. Thus we conclude that any possible nonstandard solution of \eqref{eq2.1} must own higher energy than $\mathcal{C}_{\alpha_1,\alpha_2,\beta}$ if $p\geq 2$ and $\beta\le 2^{p - 1} - 1$. However, there exists no comparison principle for such energy, which makes it quite difficult to get the synchronized phenomena existing \eqref{eq1.1} for all attractive case$(\beta>0)$. To overcome this difficulty, we derive an accurate lower bound for energy of nonstandard solutions of \eqref{eq2.1}, from which, the concentration phenomena when $\beta>0$ is small can be obtained, see \eqref{eq4.33}, \eqref{eq4.41} and \eqref{eq4.42} in Section \ref{s4} for example.

\begin{theorem}\label{th2.2}
Let $w = (u_1,u_2)$ be a nonstandard solution of \eqref{eq2.1}. Then

$(i)$ If $p\ge2$ and $0<\beta<1$, then
\begin{equation*}\label{eq2.6}
  J_{\alpha_1,\alpha_2,\beta}(w) \geq \mathcal{C}_{\alpha_1,\beta} + \mathcal{C}_{\alpha_2,\beta}.
\end{equation*}

$(ii)$ If $\beta\leq 0$, then
\begin{equation*}\label{eq2.6}
  J_{\alpha_1,\alpha_2,\beta}(w) \geq \mathcal{C}_{\alpha_1,0} + \mathcal{C}_{\alpha_2,0},
\end{equation*}
where $\mathcal{C}_{\alpha_i,\beta}$ is given in \eqref{eq1.6}.
\end{theorem}

\begin{proof}
We first consider $\beta\geq 0$.
Define the function $\mathcal{G}:(0,+\infty)\times(0,+\infty)\to \mathbb{R}$ as
$$
\mathcal{G}(t,s) := J_{\alpha_1,\alpha_2,\beta}((tu_1,su_2)).
$$
Note that $\mathcal{G}$ is continuous and $\mathcal{G}(t,s)\to-\infty$ as $(t + s)\to +\infty$. Hence $\mathcal{G}$ has at least one maximum point $\tau_{\beta} = (t_{\beta},s_{\beta})\in[0,+\infty)\times[0,+\infty)$. We claim that $\tau_{\beta}$ must be an inner point of $[0,+\infty)\times[0,+\infty)$. Firstly, note that $\tau_{\beta}\neq(0,0)$. If $t_{\beta} = 0$, then let
$g(rs_{\beta}) = \mathcal{G}(rs_{\beta},s_{\beta})$. It is easy to check that there exists $r\in(0,+\infty)$ such that $g(rs_{\beta}) > g(0)$, which is a contradiction.
Then $\nabla_{t,s}\mathcal{G}|_{\tau_{\beta}} = 0$. Denote $X_i = \|u_i\|^2_{\alpha_i}$, $Y_i = |u_i|^{2p}_{2p}$ and $Z =  |u_1u_2|^p_p$. It follows that
\begin{equation}\label{eq2.10}
\left\{
  \begin{array}{ll}
    t^2_{\beta}X_1 = t^{2p}_{\beta}Y_1 + \beta t^{p}_{\beta}s^p_{\beta}Z,&\vspace{0.12cm}\\
    s^2_{\beta}X_2 = s^{2p}_{\beta}Y_2 + \beta t^{p}_{\beta}s^p_{\beta}Z.&
  \end{array}
\right.
\end{equation}
We discuss $\tau_{\beta}$ in the following cases.

\vspace{0.2cm}
\noindent \textsf{Case 1: $p = 2$.}

\eqref{eq2.10} is equivalent to
\begin{equation*}\label{Beq2.4}
\left\{
  \begin{array}{ll}
    X_1 = t^2_{\beta}Y_1 + s^2_{\beta}\beta Z,&\vspace{0.12cm}\\
    X_2 = s^2_{\beta}Y_2 + t^2_{\beta}\beta Z,&
  \end{array}
\right.
\end{equation*}
which will have a unique solution $(1,1)$ if
\begin{equation*}
Y_1Y_2 - \beta^2Z^2 \neq 0.
\end{equation*}

\vspace{0.2cm}
\noindent
\textsf{Case 2: $p > 2$.}

It is easy to check that \eqref{eq2.10} is equivalent to
\begin{equation}\label{eq2.13}
\left\{
  \begin{array}{ll}
    (t^2_{\beta} - t^{2p}_{\beta})Y_1 = \beta Z(t^{p}_{\beta}s^p_{\beta} - t^{2}_{\beta}),&\vspace{0.12cm}\\
    (s^2_{\beta} - s^{2p}_{\beta})Y_2 = \beta Z(t^{p}_{\beta}s^p_{\beta} - s^{2}_{\beta}).&
  \end{array}
\right.
\end{equation}
Moreover,
denoting $X = \frac{s_{\beta}}{t_{\beta}}$, we have
\begin{equation}\label{Beq2.8}
 X^{2 - p}\frac{X_2}{X_1} = \frac{Y_2X^p + \beta Z}{Y_1 + \beta ZX^p}.
\end{equation}
Technically, letting $X = \Big(\frac{Y_1}{\beta Z}M\Big)^{\frac{1}{p}}$, \eqref{Beq2.8} is equivalent to
\begin{equation*}\label{}
 \Big(\frac{Y_1}{\beta Z}\Big)^{\frac{2 - p}{p}}\frac{X_2}{X_1}M^{\frac{2 - p}{p}}  = \frac{Y_2}{\beta Z} + \frac{\beta^2 Z^2 - Y_1Y_2}{\beta Z Y_1}\frac{1}{1 + M}.
\end{equation*}
Define
\begin{equation*}\label{Beq2.16}
g(M) = \Big(\frac{Y_1}{\beta Z}\Big)^{\frac{2 - p}{p}}\frac{X_2}{X_1}M^{\frac{2 - p}{p}}  - \frac{Y_2}{\beta Z} + \frac{Y_1Y_2-\beta^2 Z^2}{\beta ZY_1}\frac{1}{1 + M}.
\end{equation*}
Obviously, since $0\leq\beta<1$, by H\"older inequality, we have
\begin{equation}\label{eq2.16}
Y_1Y_2 - \beta^2Z^2 > 0.
\end{equation}
Then combining with the case that $p>2$, we conclude that $g$ decreases strictly, which implies $g$ has a unique zero point $\frac{\beta Z}{Y_1}$. Hence the $\tau_{\beta}$ in \eqref{eq2.10} must satisfy $t_{\beta} = s_{\beta}$. But, letting $t_{\beta} = s_{\beta}$ in \eqref{eq2.13}, we find
\begin{equation*}\label{}
 (t^2_{\beta} - t^{2p}_{\beta})^2Y_1Y_2 = \beta^2Z^2(t^2_{\beta} - t^{2p}_{\beta})^2,
\end{equation*}
which and \eqref{eq2.16} imply that $t_{\beta} = s_{\beta} = 1$.
Then we conclude that $\mathcal{G}$ takes maximum only at $(1,1)$.

\vspace{0.2cm}
\noindent\textsf{Proof of $(i)$.} By cases 1 and 2 above, if $p \geq 2$, we can conclude by H\"older inequality that
\begin{eqnarray*}
\begin{split}
\quad J_{\alpha_1,\alpha_2,\beta}(w)
 =  \mathcal{G}(1,1) = \max_{t,s>0}\mathcal{G}(t,s)
\geq \max_{t,s>0}\big(J_{\alpha_1,\beta}(tu_1) + J_{\alpha_2,\beta}(su_2)\big)\geq \mathcal{C}_{\alpha_1,\beta} + \mathcal{C}_{\alpha_2,\beta}.
\end{split}
\end{eqnarray*}
This proves $(i)$.

\vspace{0.2cm}

\noindent\textbf{Case 3: $\beta\leq 0$.}

We have
\begin{eqnarray*}\label{eq2.18}
\begin{split}
{\mathcal{G}}(t,s)
&= \Big(\frac{t^2}{2} - \frac{t^{2p}}{2p}\Big)X_1 + \Big(\frac{s^2}{2} - \frac{s^{2p}}{2p}\Big)X_2+ \Big(\frac{t^{2p}}{2p} + \frac{s^{2p}}{2p} - \frac{t^ps^p}{p}\Big)\beta Z\\
&\leq \Big(\frac{t^2}{2} - \frac{t^{2p}}{2p}\Big)X_1 + \Big(\frac{s^2}{2} - \frac{s^{2p}}{2p}\Big)X_2\\
&:= \widehat{\mathcal{G}}(t,s).
\end{split}
\end{eqnarray*}
It follows that
\begin{eqnarray*}\label{eq2.18}
\begin{split}
\mathcal{G}(1,1)
&\leq\max_{t,s>0}
{\mathcal{G}}(t,s)\leq \max_{t,s>0}\widehat{\mathcal{G}}(t,s) = \mathcal{G}(1,1).
\end{split}
\end{eqnarray*}
Hence $\tau_{\beta} = (1,1)$ and
\begin{equation*}
J_{\alpha_1,\alpha_2,\beta}(w)=\mathcal{G}(1,1)
\ge \max_{t,s>0}\Big(J_{\alpha_1,0}(tu_1) + J_{\alpha_2,0}(su_2)\Big)\ge \mathcal{C}_{\alpha_1,0} + \mathcal{C}_{\alpha_2,0},
\end{equation*}
which proves $(ii)$. Then the proof is completed.
\end{proof}

\vspace{0.2cm}


%

Now we  prove Theorem \ref{important}.

\begin{proof}[Proof of Theorem \ref{important}]

Let $\beta\ge 0$. It is easy to check by Changing-Of-Variable Theorem and H\"older inequality that
\begin{equation}\label{AAAeq1.21}
 \mathcal{C}_{\alpha,\beta} = \frac{\alpha^{\frac{p}{p - 1} - \frac{N}{2}}}{(1 + \beta)^{\frac{1}{p - 1}}}\mathcal{C}_{1,0},\ \ \forall\alpha>0,\ \beta\in\mathbb{R}
\end{equation}
and
\begin{equation}\label{AAAeq1.22}
\mathcal{C}_{\textsf{m}_1,\beta} + \mathcal{C}_{\textsf{m}_2,\beta}\le \mathcal{C}^*_{\textsf{m}_1,\textsf{m}_2,\beta}<\mathcal{C}_{\textsf{m}_1,0} + \mathcal{C}_{\textsf{m}_2,0}.
\end{equation}

Denote
\begin{eqnarray*}
\begin{split}
 \mathcal{F}(t,s) &= J_{\alpha_1,\alpha_2,\beta}(tU_{\alpha_1,\beta},sU_{\alpha_2,\beta}),\ \ (t,s)\in[0,+\infty)\times[0,+\infty).
 \end{split}
\end{eqnarray*}
Obviously, we can fix $T>0$ such that
\begin{equation*}
\mathcal{F}\ \text{does not take a maximum in}\ [0,+\infty)\times[0,+\infty)\backslash (0,T)\times(0,T).
\end{equation*}
Hence a maximum point $\tau_{\beta} = (t_{\beta},s_{\beta})$ of $\mathcal{F}$ must satisfy
\begin{equation*}\label{eq1.15}
\nabla_{t,s}\mathcal{F}(\tau_{\beta}) = 0.
\end{equation*}


\noindent\textbf{Case 1.} $p\geq 2$ and $\omega = 1$.

It is easy to check by the same argument in the proof of Theorem \ref{th2.2} that $\mathcal{F}$ has a unique maximum point $(1,1)$ when $0< \beta < 1$. Then
\begin{equation}\label{neweq3}
\mathcal{C}^*_{\textsf{m},\textsf{m},\beta} = \frac{2}{(1 + \beta)^{\frac{1}{p - 1}}}\mathcal{C}(\textsf{m},0).
\end{equation}
Let
$$
\tilde{\beta}_{\textsf{m},\textsf{m},p} = 1\ \ \text{and}\ \vartheta(\beta) = \bigg(\Big(\frac{2}{(1 + \beta)^{\frac{1}{p - 1}}}\Big)^{\frac{1}{\frac{p}{p - 1} - \frac{N}{2}}} - 1\bigg)\textbf{m}.
$$
It is easy to see that
if ~$0<\beta<\tilde{\beta}_{\textbf{m}_1,\textbf{m}_2,p}$ and $\delta_1,\delta_2\in [0,\vartheta(\beta)]$, then it holds
\begin{align*}
  \mathcal{C}_{\textbf{m}_1 + \delta_1,0} + \mathcal{C}_{\textbf{m}_1 + \delta_2, 0}\ge 2\mathcal{C}_{\textbf{m},0}> \frac{2}{(1 + \beta)^{\frac{1}{p - 1}}}\mathcal{C}(\textbf{m},0) = \mathcal{C}^*_{\textsf{m},\textsf{m},\beta},
\end{align*}
from which we get \eqref{eq1.4}.

\vspace{0.2cm}
\noindent\textbf{Case 2.} $\textsf{m}_1\neq\textsf{m}_2$.
\vspace{0.2cm}
Let
\begin{equation*}\label{eq1.28}
\frac{\mathcal{C}_{\textsf{m}_1,0}+\mathcal{C}_{\textsf{m}_2,0}}{\mathcal{C}_{\textsf{m}_1,0}} = \tilde{l} + \hat{l},
\end{equation*}
where $\tilde{l}\in\N\backslash\{0\}$ and $\hat{l}\in[0,1)$. By \eqref{AAAeq1.21}, we can define $\tilde{\beta}_{\textsf{m}_1,\textsf{m}_2,p}>0$ as the unique constant such that
\begin{equation*}\label{neweq4}
\frac{\mathcal{C}_{\textsf{m}_1,\tilde{\beta}_{\textsf{m}_1,\textsf{m}_2,p}} + \mathcal{C}_{\textsf{m}_2,\tilde{\beta}_{\textsf{m}_1,\textsf{m}_2,p}}}{\mathcal{C}_{\textsf{m}_1,0}}
=
\left\{
  \begin{array}{ll}
      \tilde{l} - 1, & \text{if}\ \hat{l} = 0, \vspace{0.12cm}\\
    \tilde{l}, &\textsf{if}\ \hat{l}\in(0,1).
  \end{array}
\right.
\end{equation*}
Then, for every $0<\beta<\tilde{\beta}_{\textsf{m}_1,\textsf{m}_2,p}$, defining $\vartheta(\beta)$ as the unique constant such that
\begin{eqnarray*}\label{eq1.30}
\begin{split}
\frac{\mathcal{C}_{\textsf{m}_1,\beta} + \mathcal{C}_{\textsf{m}_2,\beta}}{\mathcal{C}_{\textsf{m}_1 + \vartheta(\beta),0}} = \left\{
  \begin{array}{ll}
      \tilde{l} - 1, & \text{if}\ \hat{l} = 0, \vspace{0.12cm}\\
    \tilde{l}, &\textsf{if}\ \hat{l}\in(0,1),
  \end{array}
\right.
\end{split}
\end{eqnarray*}
we have
\begin{equation*}\label{neweq5}
  \mathcal{C}^*_{\textbf{m}_1,\textbf{m}_2,\beta} > \mathcal{C}_{\textbf{m}_1,\beta} + \mathcal{C}_{\textbf{m}_2,\beta} \ge\left\{
  \begin{array}{ll}
      (\tilde{l} - 1)\mathcal{C}_{\textbf{m}_1 + \vartheta(\beta),0}, & \text{if}\ \hat{l} = 0, \vspace{0.12cm}\\
    \tilde{l}\mathcal{C}_{\textbf{m}_1 + \vartheta(\beta),0}, &\text{if}\ \hat{l}\in(0,1).
  \end{array}
\right.
\end{equation*}
Hence, when $\hat{l} = 0$,  for all $\delta_1,\ldots,\delta_{\tilde{l}}\in[0,\vartheta(\beta))$, by \eqref{AAAeq1.22}, we have
\begin{equation}\label{addeq2.1}
\left\{
  \begin{array}{ll}
    \sum_{i = 1}^{\tilde{l}-1}\mathcal{C}_{\textbf{m}_1+\delta_i,0}< (\tilde{l} - 1)\mathcal{C}_{\textbf{m}_1+\vartheta(\beta),0}=\mathcal{C}_{\textsf{m}_1,\beta} + \mathcal{C}_{\textsf{m}_2,\beta}, & \\
     \sum_{i = 1}^{\tilde{l}}\mathcal{C}_{\textbf{m}_1+\delta_i,0}\ge \tilde{l}\mathcal{C}_{\textbf{m}_1,0}>\mathcal{C}^*_{\textbf{m}_1,\textbf{m}_2,\beta}.
  \end{array}
\right.
\end{equation}
When $\hat{l}\in(0,1)$, by the same reason above, we conclude that for all $\delta_1,\ldots,\delta_{\tilde{l}+1}\in[0,\vartheta(\beta))$, it holds
\begin{equation}\label{addeq2.2}
\left\{
  \begin{array}{ll}
    \sum_{i = 1}^{\tilde{l}}\mathcal{C}_{\textbf{m}_1+\delta_i,0}< \tilde{l}\mathcal{C}_{\textbf{m}_1+\vartheta(\beta),0}=\mathcal{C}_{\textsf{m}_1,\beta} + \mathcal{C}_{\textsf{m}_2,\beta}, & \\
     \sum_{i = 1}^{\tilde{l}+1}\mathcal{C}_{\textbf{m}_1+\delta_i,0}\ge (\tilde{l}+1)\mathcal{C}_{\textbf{m}_1,0}>\mathcal{C}^*_{\textbf{m}_1,\textbf{m}_2,\beta}.
  \end{array}
\right.
\end{equation}
Obviously, $\vartheta(\cdot):(0,\tilde{\beta}_{\textsf{m}_1,\textsf{m}_2,p})\to (0,+\infty)$ is continuous and decreasing, from which, \eqref{addeq2.1} and \eqref{addeq2.2} we  completes the proof of Theorem \ref{important}.

\end{proof}
%

\section{The penalized scheme}\label{s3}

\subsection{The penalized functional}
Let $f(x)\in C^{\infty}(\mathbb{R}^N)$ satisfy

\begin{align*}
  f_{\alpha}(x): = \left\{
            \begin{array}{ll}
              1, & x\in B_1(0), \vspace{2mm}\\
              \in[0,1],& x\in B_2(0)\backslash B_1(0),\vspace{2mm}\\
              \frac{1}{|x|^{\alpha}}, & x\in \mathbb{R}^N\backslash B_2(0).
            \end{array}
          \right.
\end{align*}
It is easy to check that if $V_{\min}(x)=O(|x|^{-2\hat{\sigma}})$ with $\hat{\sigma} > 0$, then for $\alpha\in \big(\max\{\frac{N - 2\hat{\sigma}}{2},\frac{N}{2}\},\frac{N}{2p}\big)$, it holds
$$
\|f_{\alpha}\|_{\mathcal{H}_{\epsilon}}<+\infty\ \text{but}\ |f_{\alpha}|^{2p}_{2p} = +\infty.
$$
Hence we need to cut off the nonlinear terms ``$|u_i|^{2p - 2}u_i,\,\,|u_i|^{p - 2}u_i|u_j|^p$" $(i,j = 1,2).$


We choose a family of penalized potentials $P_{\epsilon}\in L^{\infty}(\mathbb{R}^N,[0,\infty))$ such that
\begin{eqnarray}\label{eq3.1}
\begin{split}
 P_{\epsilon}(x) = 0\,\,\text{in}\,\,\Lambda\,\,\ \text{and}
\left\{
 \begin{array}{ll}
    \limsup_{\epsilon\to 0}\sup_{x\in\mathbb{R}^N\backslash\Lambda}P_{\epsilon}(x)|x|^{(2 + \kappa)\sigma} = 0, & \text{if}\ \eqref{AAeq1.8}\ \text{holds}\\
    \limsup_{\epsilon\to 0}\epsilon^{-2}\sup_{x\in\mathbb{R}^N\backslash\Lambda}P_{\epsilon}(x)|x|^{2} = 0, & \text{if}\ \eqref{AAeq1.9}\ \text{holds},
  \end{array}
\right.
\end{split}
\end{eqnarray}
where $\kappa> 0$ is a small parameter.


By the prior decay assumption on $P_{\epsilon}$,
we immediately have the following proposition:

\begin{proposition}\label{pr3.1}(\cite[Theorem 4]{30}, \cite[Lemma 3.5]{MV}).
The embedding $H^1_{V_i,\epsilon}(\mathbb{R}^N)\subset\subset L^2(\mathbb{R}^N, (P_{\epsilon}(x) + \chi_{\Lambda}(x))dx)$ is compact. Moreover, for every $\tau>0$, there exists an $\epsilon_{\tau}>0$ such that if $\epsilon\in(0,\epsilon_{\tau})$,
$$
\int_{\mathbb{R}^N}P_{\epsilon}(x)|\varphi_i|^2 dx \leq \tau \int_{\mathbb{R}^N}\epsilon^{2}|\nabla  \varphi|^2 + V_i|\varphi_i|^2
$$
for each $\varphi_i\in H^1_{V_i,\epsilon}(\mathbb{R}^N)$ $(i = 1,2)$.
\end{proposition}
\begin{proof}
When \eqref{AAeq1.8} holds, the proof is given in \cite[Theorem 4]{30}. When \eqref{AAeq1.8} does not hold, i.e., $V_{min}$ vanishes faster than $|x|^{-2}$ or has compact support, the proof is given in \cite[Lemma 3.5]{MV},  which is based on the well-known Hardy inequality: When $N\ge 3$, it holds
\begin{equation}\label{CCeq3.3}
  \int_{\mathbb{R}^N}|\nabla u|^2\geq \frac{(N - 2)^2}{4}\int_{\mathbb{R}^N}\frac{|u|^2}{|x|^2}\ \ \text{for all}\ u\in D^{1,2}(\R^N).
\end{equation}
\end{proof}

\begin{remark}
The idea of construction a penalized function when \eqref{AAeq1.8} does not hold is also from \eqref{CCeq3.3}, which and a standard variational argument imply the operator
\begin{equation*}\label{}
  -\epsilon^2\Delta  - \epsilon^2 \theta|x|^{-2} + V_{min}
\end{equation*}
is positive if $\theta < \frac{(N - 2)^2}{4}$. See  case 2 in Section \ref{s5} for more details.
\end{remark}

Given a penalized potential $P_{\epsilon}$ that satisfies  \eqref{eq3.1},
we define the penalized nonlinearities $g_{ \epsilon}:\mathbb{R}^N\times \mathbb{R}\to \mathbb{R}$  by
$$
g_{\epsilon}(x,s): = \chi_{\Lambda}(x)s^{2p - 1}_{+} + \chi_{\mathbb{R}^N\backslash \Lambda}(x)\min(s^{2p - 1}_{+}, P_{\epsilon}(x)s_+).
$$
We also denote $G_{\epsilon}(x,t) = \int_{0}^{t}g_{ \epsilon}(x,s)ds$. Moreover, we define
$$
\tilde{g}_{\epsilon}(x,s): = \chi_{\Lambda}(x)s^{p - 1}_{+} + \chi_{\mathbb{R}^N\backslash \Lambda}(x)\min \Big(s^{p - 1}_{+}, \sqrt{P_{\epsilon}(x)}/\sqrt{p} \Big)
$$
and denote $\widetilde{G}_{\epsilon}(x,t) = \int_{0}^{t}\tilde{g}_{\epsilon}(x,s)ds$.
Accordingly, we define the penalized superposition operators $\mathfrak{g}_{\epsilon}$, $\mathfrak{G}_{\epsilon}$,  $\widetilde{\mathfrak{g}}_{\epsilon}$ and $\widetilde{\mathfrak{G}}_{\epsilon}$  by
\begin{align*}
&\mathfrak{g}_{\epsilon}(u)(x) = g_{\epsilon}(x,u(x)),\ \mathfrak{G}_{\epsilon}(u)(x) = G_{\epsilon}(x,u(x)),\\ &\widetilde{\mathfrak{g}}_{\epsilon}(u)(x) = \tilde{g}_{\epsilon}(x,u(x)),\ \widetilde{\mathfrak{G}}_{\epsilon}(u)(x) = \widetilde{G}_{\epsilon}(x,u(x))
\end{align*}
and the penalized functional $J_{\epsilon}:\mathcal{H}_{\epsilon}\to \mathbb{R}$ by
\begin{align*}
  J_{\epsilon}(w) &= \frac{1}{2}\|w\|^2_{\mathcal{H}_{\epsilon}}
                    - \int_{\mathbb{R}^N}\mathfrak{G}_{\epsilon}(u_1) + \mathfrak{G}_{\epsilon}(u_2) - p\beta\int_{\mathbb{R}^N} \widetilde{\mathfrak{G}}_{\epsilon}(u_2)\widetilde{\mathfrak{G}}_{\epsilon}(u_1)dx.
\end{align*}

Now with the help of \eqref{eq3.1} and Proposition \ref{pr3.1}, we are going to prove the following lemma, which says that $J_{\epsilon}\in C^1(\mathcal{H}_{\epsilon},\R)$ and satisfies (P.S.) condition. It is a basic requirement for looking
for solutions. The proof is not obvious because of the coupling effect.

\begin{lemma}\label{le3.2}

(i) Let $2 < p < 2^*$. Then $J_{\epsilon}\in C^1(\mathcal{H}_{\epsilon}, \mathbb{R})$ and $w = (u_1,u_2)\in \mathcal{H}_{\epsilon}$ is a critical point of $J_{\epsilon}$ if and only if $w$ is a weak solution of the penalized system:
\begin{equation}\label{Ceq3.3}
\left\{
  \begin{array}{ll}
   - \epsilon^{2}\Delta  u_1 + V_1(x)u_1 = \mathfrak{g}_{\epsilon}(u_1) + p\beta\tilde{\mathfrak{g}}_{\epsilon}(u_1)\widetilde{\mathfrak{G}}_{\epsilon}(u_2), & x\in \mathbb{R}^N,\vspace{0.12cm}\\
    -\epsilon^{2}\Delta  u_2 + V_2(x)u_2 = \mathfrak{g}_{\epsilon}(u_2) + p\beta\tilde{\mathfrak{g}}_{\epsilon}(u_2)\widetilde{\mathfrak{G}}_{\epsilon}(u_1), & x\in \mathbb{R}^N.
  \end{array}
\right.
\end{equation}

(ii) (P.S. condition) $J_{\epsilon}$ satisfies the Palais-Smale condition if $0<\epsilon<\epsilon_0$ for some $\epsilon_0 > 0$.

\end{lemma}

\begin{proof}
For simplicity, we only show the term
$$
\mathcal{N}_{\epsilon}(w) = \int_{\mathbb{R}^N}\widetilde{\mathfrak{G}}_{\epsilon}(u_1)\widetilde{\mathfrak{G}}_{\epsilon}(u_2)
$$
belongs to $C^1(\mathcal{H}_{\epsilon},\R)$, since the other terms are similar.

Firstly, fixing every $\varphi= (\varphi_1,\varphi_2)\in \mathcal{H}_{\epsilon}$, for all $t\in\mathbb{R}$ with $|t| \leq 1$, by the triangle inequality, it holds
\begin{align*}
&\qquad\mathcal{N}_{\epsilon}(w + t\varphi) - \mathcal{N}_{\epsilon}(w)\\
&\leq C\Big(\chi_{\Lambda}(|u_1|^{2p} + |u_2|^{2p}  + |\varphi_1|^{2p} + |\varphi_2|^{2p})+ \chi_{\mathbb{R}^N\backslash\Lambda}P_{\epsilon}(|u_1|^{2} + |u_2|^{2}  + |\varphi_1|^{2} + |\varphi_2|^{2})\Big)\\
&\in L^1(\mathbb{R}^N).
\end{align*}
Then the existence of the first Gateaux derivative follows by Dominated Convergence Theorem and Proposition \ref{pr3.1}.

Secondly, given any $\varphi = (\varphi_1,\varphi_2)\in \mathcal{H}_{\epsilon}$ with $\|\varphi\|_{\epsilon}\leq 1$ and $w_n = (u^1_n,u^2_n)\in \mathcal{H}_{\epsilon}$ with $w_n\to w = (u_1,u_2)$ in $\mathcal{H}_{\epsilon}$, by Proposition \ref{pr3.1} and H\"{o}lder inequality, we have
\begin{align*}
 &\qquad |\langle \mathcal{N}'_{\epsilon}(w_n) - \mathcal{N}'_{\epsilon}(w),\varphi\rangle|\\
 &\leq\Big|\int_{\mathbb{R}^N}\tilde{\mathfrak{g}}_{\epsilon}(u^1_n)\varphi_1\widetilde{\mathfrak{G}}_{\epsilon}(u^2_n)  - \tilde{\mathfrak{g}}_{\epsilon}(u_1)\varphi_1\widetilde{\mathfrak{G}}_{\epsilon}(u_2)\Big|+ \Big|\int_{\mathbb{R}^N}\tilde{\mathfrak{g}}_{\epsilon}(u^2_n) \varphi_2\widetilde{\mathfrak{G}}_{\epsilon}(u^1_n) - \tilde{\mathfrak{g}}_{\epsilon}(u_2)\varphi_2\widetilde{\mathfrak{G}}_{\epsilon}(u_1)\Big|\\
 &: = I^1_n + I^2_n.
\end{align*}
For $I^1_n$, we have
\begin{align*}
I^1_n
& = \Big|\int_{\mathbb{R}^N}\tilde{\mathfrak{g}}_{\epsilon}(u^1_n)\varphi_1\widetilde{\mathfrak{G}}_{\epsilon}(u^2_n)  - \tilde{\mathfrak{g}}_{\epsilon}(u_1)\varphi_1\widetilde{\mathfrak{G}}_{\epsilon}(u_2)\Big|\\
&\leq \int_{\mathbb{R}^N}|\widetilde{\mathfrak{G}}_{\epsilon}(u^2_n)  - \widetilde{\mathfrak{G}}_{\epsilon}(u_2)||\tilde{\mathfrak{g}}_{\epsilon}(u^1_n)|\varphi_1|+ \int_{\mathbb{R}^N}|\tilde{\mathfrak{g}}_{\epsilon}(u^1_n) - \tilde{\mathfrak{g}}_{\epsilon}(u_1)||\varphi_1||\widetilde{\mathfrak{G}}_{\epsilon}(u_2)|\\
&: = I^{11}_n + I^{12}_n.
\end{align*}

By Dominated Convergence Theorem, Mean Value Theorem, Sobolev embedding theorem and \eqref{eq3.1}, we have
\begin{align*}
I^{11}_n
&\leq \|\tilde{g}_{\epsilon}(u^1_n)\varphi_1\|_{L^2(\mathbb{R}^N)}\Big(\int_{\mathbb{R}^N}|\widetilde{\mathfrak{G}}_{\epsilon}(u^2_n)  - \widetilde{\mathfrak{G}}_{\epsilon}(u_2)|^2\Big)^{\frac{1}{2}}\\
&\leq C\Big(\|u^2_n - u_2\|^{2p}_{H^1_{V,\epsilon}(\mathbb{R}^N)}
 + \int_{\Lambda^c} |\min\{\sqrt{P_{\epsilon}}, ((u^2_n)_++(u_2)_+)^{p-1}\}((u^2_n) _+ - (u_2)_+)|^2\\
&\leq C\Big(\|u^2_n - u_2\|^{2p}_{H^1_{V,\epsilon}(\mathbb{R}^N)} + \|u^2_n - u_2\|^{2}_{H^1_{V,\epsilon}(\mathbb{R}^N)} + o_n(1)\Big)^{\frac{1}{2}}\\
& = o_n(1)
\end{align*}
and
\begin{align*}
I^{12}_n
& = o_n(1).
\end{align*}
Similarly, we have $I^2_n = o_n(1)$. This completes the proof of $(i)$.

Next, we prove $(ii)$. Our aim is to verify that every sequence $(w_n) = (u^1_n,u^2_n)\in \mathcal{H}_{\epsilon}$ satisfies

(i) $J'_{\epsilon}(w_n)\to 0$ in $\mathcal{H}_{\epsilon}';$

(ii) $\sup_{n}J_{\epsilon}(w_n)\leq C < +\infty$

\noindent is relatively compact.

Firstly, by Proposition \ref{pr3.1}, there exists an $\epsilon_0>0$ such that if $0<\epsilon<\epsilon_0$,
\begin{align*}
c\|w_n\|^2_{\epsilon}&\leq\big(\frac{1}{2} - \frac{1}{2p}\big)\|w_n\|^2_{\epsilon} - \int_{\mathbb{R}^N\backslash\Lambda}\mathfrak{G}^1_{\epsilon}(u^1_n)- \frac{1}{2p} \min\{P_{\epsilon}(u^1_n)^2_+,(u^1_n)^{2p}_+\}\\
&\quad - \int_{\mathbb{R}^N\backslash\Lambda}\mathfrak{G}^2_{\epsilon}(u^2_n) - \frac{1}{2p} \min\{P_{\epsilon}(u^2_n) ^2_+,(u^2_n) ^{2p}_+\}\\
 &\quad+ \big(\frac{1}{2p} - \frac{1}{2p}\big)\int_{\Lambda}(u^1_n)^{2p}_+ + \big(\frac{1}{2p} - \frac{1}{2p}\big)\int_{\Lambda}(u^2_n) ^{2p}_+\\
&\quad - \int_{\mathbb{R}^N}p\beta\widetilde{\mathfrak{G}}_{\epsilon}(u^1_n)\widetilde{\mathfrak{G}}_{\epsilon}(u^2_n)  - \frac{\beta}{2}\widetilde{\mathfrak{g}}_{\epsilon}(u^1_n)u^1_n\widetilde{\mathfrak{G}}_{\epsilon}(u^2_n)  - \frac{\beta}{2}\widetilde{\mathfrak{g}}_{\epsilon}(u^2_n) u^2_n\widetilde{\mathfrak{G}}_{\epsilon}(u^1_n)\\
&\leq C + \frac{c}{2}\|u_n\|^2_{\epsilon}.
\end{align*}
Hence $\{w_n\}$ is bounded in $\mathcal{H}_{\epsilon}$ if $0<\epsilon<\epsilon_0$.

Now, going if necessary to a subsequence, we assume that $w_n = (u^1_n,u^2_n)\rightharpoonup w = (u_1,u_2)\in  \mathcal{H}_{\epsilon}$. By Proposition \ref{pr3.1} and Dominated Convergence Theorem, it holds
\begin{align*}
  \|w_n - w\|^2_{\mathcal{H}_{\epsilon}}
& \leq |\langle J'_{\epsilon}(w_n), w_n - w\rangle| + |\langle J'_{\epsilon}(w) - J'_{\epsilon}(w_n) ,w\rangle|\\
&\quad + \int_{\mathbb{R}^N}\mathfrak{g}_{\epsilon}(u^1_n)|u^1_n - u_1| + |\mathfrak{g}_{\epsilon}(u_1) - \mathfrak{g}_{\epsilon}(u^1_n)||u_1|\\
&\quad + \int_{\mathbb{R}^N}\mathfrak{g}_{\epsilon}(u^2_n) |u^2_n - u_2| + |\mathfrak{g}_{\epsilon}(u_2) - \mathfrak{g}_{\epsilon}(u^2_n) ||u_2|\\
&\quad + p\beta\int_{\mathbb{R}^N}\widetilde{\mathfrak{g}}_{\epsilon}(u^1_n)|u^1_n - u_1|\widetilde{\mathfrak{G}}_{\epsilon}(u^2_n) \\
&\quad + p\beta\int_{\mathbb{R}^N}|\widetilde{\mathfrak{g}}_{\epsilon}(u^1_n) - \widetilde{\mathfrak{g}}_{\epsilon}(u_1)||u_1|\widetilde{\mathfrak{G}}_{\epsilon}(u^2_n) \\
&\quad + p\beta\int_{\mathbb{R}^N}\widetilde{\mathfrak{g}}_{\epsilon}(u^2_n) |u^2_n - u_2|\widetilde{\mathfrak{G}}_{\epsilon}(u^1_n)\\
&\quad + p\beta\int_{\mathbb{R}^N}|\widetilde{\mathfrak{g}}_{\epsilon}(u^2_n)  - \widetilde{\mathfrak{g}}_{\epsilon}(u_2)||u_2|\widetilde{\mathfrak{G}}_{\epsilon}(u^1_n)\\
& = o_n(1).
\end{align*}
Then  $(w_n)$ is relatively compact in $\mathcal{H}_{\epsilon}$ and the conclusion follows.
\end{proof}

After showing that $J_{\epsilon}\in C^1(\mathcal{H}_{\epsilon},\R)$, we are going to prove the existence of nonstandard solutions to the penalized problem \eqref{Ceq3.3}. When $p\ge 2$ and $\beta>0$ is small, we have to construct nonstandard solutions with higher energy. For this purpose, we construct
 skillfully a two dimensional mountain path geometry and use \cite[Theorem 2.8]{MW} to find such solutions. We need the conclusion in Theorem \ref{important} in such
 an approach.

\begin{definition}\label{de3.5}
We say a path $\gamma\in C([0,1]^2,\mathcal{H}_{\epsilon})$ belongs to ${\Gamma}_{\epsilon}$ if
$$
\gamma({\tau}) = (tTU^{\epsilon}_{\textsf{m}_1,\beta},sTU^{\epsilon}_{\textsf{m}_2,\beta}),\ \ \forall\tau = (t,s)\in\partial[0,1]^2,
$$
where $U^{\epsilon}_{\textsf{m}_i,\beta}(\cdot) = U_{\textsf{m}_i,\beta}\Big(\frac{\cdot - \textsf{p}}{\epsilon}\Big)$ with $\textsf{p}\in\mathcal{M}$, $T> 0$ is a suitably large constant such that
\begin{equation*}\label{}
  J^i_{\epsilon}(TU^{\epsilon}_{\textsf{m}_i,\beta}) < 0,\ \ i = 1,2,
\end{equation*}
$J^i_{\epsilon}:H^1_{V_i,\epsilon}(\mathbb{R}^N)\to \mathbb{R}$ is defined as
$$
J^i_{\epsilon}(u) = \frac{1}{2}\|u\|^2_{V_i,\epsilon} - \int_{\mathbb{R}^N}\mathfrak{G}_{\epsilon}(u)-\frac{p\beta}{2}\int_{\R^N}(\widetilde{\mathfrak{G}}_{\epsilon}(u))^2.
$$
\end{definition}

\begin{lemma}\label{adle}
The mountain pass value
\begin{equation}\label{Ceq3.7}
  {c}^i_{\epsilon} = \inf_{{\Gamma}^i_{\epsilon}}\max_{t\in[0,1]}J^i_{\epsilon}(\gamma(t))
\end{equation}
can be achieved, where
$$
{\Gamma}^i_{\epsilon} = \{\gamma\in C([0,1],H^1_{V_i,\epsilon}(\mathbb{R}^N)):\gamma(0) = 0,\ \, J^i_{\epsilon}(\gamma(1))<0\}.
$$
Moreover, we can conclude that
\begin{equation}\label{eq3.7}
  \lim_{\epsilon\to 0}\frac{{c}^i_{\epsilon}}{\epsilon^N} =\mathcal{C}_{\textsf{m}_i,\beta},
\end{equation}
where $\mathcal{C}_{\textsf{m}_i,\beta}$ and $U_{\textsf{m}_i,\beta}$ are given in\eqref{eq1.6}.
\end{lemma}
We do not give the proof since it is the same as the estimates in Lemma \ref{le4.4} in the next section.

Let
\begin{equation}\label{CCeq3.9}
\beta_{\omega,p} = (1 + \omega^{\frac{p}{p - 1} - \frac{N}{2}})^{p - 1} - 1.
\end{equation}
Then, by the special choice of $\Gamma_{\epsilon}$ and Lemma \ref{adle}, we have:
\begin{lemma}\label{le3.8}
If $0<\beta<\beta_{\omega,p}$, then the mountain pass value
$$
{\mathcal{C}}_{\epsilon}: = \inf_{\gamma\in{\Gamma}_{\epsilon}}\max_{\tau\in[0,1]^2}J_{\epsilon}(\gamma(\tau))
$$
can be achieved by a function $w_{\epsilon}\in\mathcal{H}_{\epsilon}$ which solves the penalized equation \eqref{Ceq3.3}; moreover,
\begin{eqnarray}\label{eq3.9}
\begin{split}
&\mathcal{C}^*_{\textbf{m}_1,\textbf{m}_2,\beta}\ge \limsup_{\varepsilon\to 0}\frac{{\mathcal{C}}_{\varepsilon}}{\varepsilon^N}
\ge\liminf_{\varepsilon\to 0}\frac{{\mathcal{C}}_{\varepsilon}}{\varepsilon^N}\ge \mathcal{C}_{\textbf{m}_1,\beta} + \mathcal{C}_{\textbf{m}_2,\beta}
> \mathcal{C}_{\textbf{m}_2,0}\\
&\qquad\geq \limsup_{\varepsilon\to 0}\varepsilon^{-N}\sup_{\gamma\in{\Gamma}_{\varepsilon}}\max_{\tau\in \partial[0,1]^2}J_{\varepsilon}(\gamma(\tau)).
\end{split}
\end{eqnarray}
where $\mathcal{C}^*_{\textsf{m}_1,\textsf{m}_2,\beta}$ is defined in \eqref{eq1.7}.
\end{lemma}
\begin{proof}
The upper bound follows by letting $(tTU^{\epsilon}_{\textsf{m}_1,\beta},sTU^{\epsilon}_{\textsf{m}_2,\beta})$ be a special path.

For each $\gamma\in\widetilde{\Gamma_{\varepsilon}}$, assuming that $\gamma(\tau) = (\gamma_1(\tau),\gamma_2(\tau))$, observing that for each continuous map $c:[0,1]\to [0,1]^2$ with $c(0)\in \{0\}\times[0,1]$ and $c(1)\in \{1\}\times [0,1]$, it holds
\begin{align*}
  J^1_{\epsilon}(\gamma_1(c(0))) =  0\ \text{and}\ J^1_{\epsilon}(\gamma_1(c(1))) < 0.
\end{align*}
Hence $\gamma_1(c(s))\in {\Gamma}^1_{\epsilon}$, which implies
$$
\max_{s\in[0,1]}J^1_{\epsilon}(\gamma_1(c(s)))\geq {c}^1_{\epsilon},
$$
where ${c}^i_{\epsilon}$ is given in \eqref{Ceq3.7}.
Similarly, letting $c:[0,1]\to [0,1]^2$ be a continuous map with $c(0)\in [0,1]\times\{0\}$ and $c(1)\in [0,1]\times\{1\}$, it holds
$$
\max_{t\in [0,1]}J^2_{\epsilon}(\gamma_2(c(t))\geq {c}^2_{\epsilon}.
$$
Now using the same argument as that of Proposition 3.4 in \cite{CP1}, we find  a $\hat{\tau}\in [0,1]^2$ such that
\begin{equation*}
  J^1_{\epsilon}(\gamma_1(\hat{\tau}))\geq {c}^1_{\epsilon},\ J^2_{\epsilon}(\gamma_2(\hat{\tau}))\geq {c}^2_{\epsilon}.
\end{equation*}
Thus by H\"{o}lder inequality and \eqref{eq3.7}, we have
\begin{align*}
  &\quad  \liminf_{\epsilon\to 0}\frac{1}{\epsilon^N}\max_{\tau\in[0,1]^2}J_{\epsilon}(\gamma(\tau))
\geq \liminf_{\epsilon\to 0}\frac{1}{\epsilon^N}\big({c}^1_{\epsilon} + {c}^2_{\epsilon}\big)= \mathcal{C}_{\textsf{m}_1,\beta} + \mathcal{C}_{\textsf{m}_2,\beta}> \mathcal{C}_{\textbf{m}_2,0}
\end{align*}
if $0<\beta<\beta_{\omega,p}$.
The lower bound then follows.

Now, by \eqref{eq3.9}, Theorem 2.8 in \cite{MW} and Lemma \ref{le3.2}, there exists a sequence $\{w_n = (u^1_n,u^2_n):n\in\mathbb{N}\}\subset \mathcal{H}_{\epsilon}$ converging strongly to $w_{\epsilon}$ in $\mathcal{H}_{\epsilon}$ such that
\begin{equation*}\label{PS-sequence}
 J_{\epsilon}(w_{\epsilon}) = {\mathcal{C}}_{\epsilon}\ \text{and}\ J'_{\epsilon}(w_{\epsilon}) = 0.
\end{equation*}
This completes the proof.
\end{proof}

At last, we show that the penalized solution given by $w_{\epsilon}$ Lemma \ref{le3.8} is nonstandard. This procedure is more difficult than the ground case. We observe firstly what will happen if $w_{\epsilon}$ is standard. We have the following Concentration-Compactness Lemma(see \cite{CZ} for the case $\epsilon = 1$):
\begin{lemma}\label{le3.9}(Concentration-Compactness Lemma)
If
$$
\lim_{\epsilon\to 0}\|{u}^l_{\epsilon}\|_{L^{\infty}(\Lambda)} = 0,
$$
then there exists $k_j>1,k_j\in\mathbb{N}$ and $\delta_1,\ldots,\delta_{k_j}\in [0, \sup_{x\in\Lambda}V_j(x) - \textbf{m}_j]$ such that
\begin{equation*}\label{eq3.11}
\lim_{\epsilon\to 0}\frac{{\mathcal{C}}_{\epsilon}}{\epsilon^N} =  \sum_{i = 1}^{k_j}\mathcal{C}_{\textsf{m}_j+\delta_i,0},
\end{equation*}
where $1\leq l\neq j\leq 2$.
\end{lemma}

For the sake of simplicity, we give the proof in the next section.

It is easy to see that there is a contradiction between \eqref{eq1.4}, \eqref{reviseeq} and \eqref{eq3.11}. Hence, rearranging $\tilde{\beta}_{\textsf{m}_1,\textsf{m}_2,p}<\beta_{\omega,p}$, where $\tilde{\beta}_{\textsf{m}_1,\textsf{m}_2,p}$ is the constant in Theorem \ref{important}, it immediately holds:
\begin{lemma}\label{Cle3.8}
Let \eqref{reviseeq} hold. Then the critical points $\{w_{\epsilon}:0<\epsilon<\epsilon_0\}$ obtained by Lemma \ref{le3.8} are nonstandard, i.e.,
\begin{equation*}\label{eq3.13}
  \liminf_{\epsilon\to 0}\|{u}^1_{\epsilon}\|_{L^{\infty}(\R^N)} > 0\ \ \text{and}\ \liminf_{\epsilon\to 0}\|{u}^2_{\epsilon}\|_{L^{\infty}(\R^N)} > 0.
\end{equation*}
\end{lemma}


\section{Concentration of the penalized solutions}\label{s4}

In this section we prove Lemma \ref{le3.9} and the concentration phenomenon of the penalized solution $w_{\epsilon}$ in Lemma \ref{le3.8}. Some blow-up analysis of Lemma \ref{le3.9} are similar to the proof of the concentration phenomenon of the penalized solution $w_{\epsilon}$. For this reason, we first give the proof of Lemma \ref{le3.9}.

Before to prove Lemma \ref{le3.9}, we need a Liouville type theorem for systems on a half-space.

\begin{lemma}\label{le4.2}
Let $\beta>-1$ and $H\subset \mathbb{R}^N$ be a half-plane. If $u_1,u_2\geq 0$ satisfy the following system
\begin{align*}
\left\{
  \begin{array}{ll}
   - \Delta u_1 + \alpha_1 u_1= \chi_{H}u_1^{2p - 1} + \beta \chi_{H}u_1^{p - 1}u_2^p, & x\in \mathbb{R}^N,\vspace{0.12cm}\\
   -\Delta u_2 + \alpha_2 u_2= \chi_{H}u_2^{2p - 1}  + \beta \chi_{H}u_2^{p - 1}u_1^p, & x\in \mathbb{R}^N,
  \end{array}
\right.
\end{align*}
then $(u_1,u_2) = (0,0)$.
\end{lemma}

\begin{proof}
Without loss of generality, we assume $H = \mathbb{R}^{N - 1}\times(0,\infty)$.
By the classical regularity argument in \cite{GT}, we have $u_i\in H^2(\mathbb{R}^N)$. Testing the equation against $(\partial_N u_1,\partial_Nu_2)$, we find
\begin{align*}
\left\{
  \begin{array}{ll}
   0 = \int_{\mathbb{R}^{N - 1}\times (0,\infty)}u_1^{2p - 1}\partial_N u_1 + \beta\int_{\mathbb{R}^{N - 1}\times (0,\infty)}u_1^{p - 1}u_2^p\partial_N u_1, & x\in \mathbb{R}^N,\vspace{0.12cm}\\
   0 = \int_{\mathbb{R}^{N - 1}\times (0,\infty)}u_2^{2p - 1}\partial_N u_2  + \beta\int_{\mathbb{R}^{N - 1}\times (0,\infty)} u_2^{p - 1}u_1^p\partial_N u_2, & x\in \mathbb{R}^N.
  \end{array}
\right.
\end{align*}
It follows that
\begin{align*}
\frac{1}{2p}\int_{\mathbb{R}^{N - 1}}\big(u_1^{2p}(x',0) + u_2^{2p}(x',0)\big)dx' + \frac{\beta}{p}\int_{\mathbb{R}^{N - 1}}u_1^{p}(x',0)u_2^{p}(x',0)dx'
=0.
\end{align*}
So $u_1(x),u_2(x) = 0$ on $\partial\mathbb{R}^N_+$ if $\beta>-1$. Then by the
strong maximum principle we have $u_1(x),u_2(x) \equiv 0.$
\end{proof}

\textbf{Proof of Lemma \ref{le3.9}.}
Without loss of generality, we assume that there exists a sequence of $(\epsilon_n)$, $\epsilon_n > 0$ with $\lim_{n\to\infty}\epsilon_n = 0$ such that
\begin{equation}\label{eq4.17}
\lim_{n\to\infty}\|u^1_{\epsilon_n}\|_{L^{\infty}(\Lambda)} = 0.
\end{equation}
But, testing \eqref{Ceq3.3} against with $w_{\epsilon_n}$, by  \eqref{eq3.1}, we find
\begin{eqnarray*}
\lim_{n\to\infty}\|u^1_{\epsilon_n} + u^2_{\epsilon_n}\|_{L^{\infty}(\Lambda)} > 0.
\end{eqnarray*}
Hence
\begin{equation*}
  \liminf_{n\to\infty}\|u^2_{\epsilon_n}\|_{L^{\infty}(\Lambda)} > 0.
\end{equation*}
Then, by standard regularity argument in \cite{GT}, we assume that there is a $x^2_{\epsilon_n}\in\overline{\Lambda}$ such that
$$
u^2_{\epsilon_n}(x^2_{\epsilon_n}) = \sup_{\Lambda}u^2_{\epsilon_n}(x).
$$
Let $\tilde{w}_{\epsilon_n}(\cdot) = \big(u^1_{\epsilon_n}(\epsilon_n \cdot + x^2_{\epsilon_n}),u^2_{\epsilon_n}(\epsilon_n \cdot + x^2_{\epsilon_n})\big) = (\tilde{u}^1_{\epsilon_n}(\cdot),\tilde{u}^2_{\epsilon_n}(\cdot))$, $\widetilde{V}^i_{\epsilon_n}(\cdot) = V_i(\epsilon_n \cdot + x^2_{\epsilon_n})$, $\hat{g}_{\epsilon_n}(\cdot) = g_{\epsilon_n}(\epsilon_n x + x^2_{\epsilon_n},\cdot)$, $\widehat{G}_{\epsilon_n}(\cdot) = G_{\epsilon_n}(\epsilon_n x + x^2_{\epsilon_n},\cdot)$, $\breve{g}_{\epsilon_n}(\cdot) = \tilde{g}_{\epsilon_n}(\epsilon_n x + x^2_{\epsilon_n},\cdot)$ and $\breve{G}_{\epsilon_n}(\cdot) = \widetilde{G}_{\epsilon_n}(\epsilon_n x + x^2_{\epsilon_n},\cdot)$.
It is easy to check $\tilde{w}_{\epsilon_n}$ satisfies
\begin{align}\label{eq4.6}
\left\{
  \begin{array}{ll}
    -\Delta \tilde{u}^1_{\epsilon_n} + \tilde{V}^1_{\epsilon_n}(x)\tilde{u}^1_{\epsilon_n} = \hat{g}_{\epsilon_n}(\tilde{u}^1_{\epsilon_n}) + p\beta\breve{g}_{\epsilon_n}(\tilde{u}^1_{\epsilon_n})
\breve{G}_{\epsilon_n}(\tilde{u}^2_{\epsilon_n}), &  \text{in}\,\,\mathbb{R}^N,\vspace{0.12cm}\\
    -\Delta \tilde{u}^2_{\epsilon_n} + \tilde{V}^2_{\epsilon_n}(x)\tilde{u}^2_{\epsilon_n} = \hat{g}_{\epsilon_n}(\tilde{u}^2_{\epsilon_n}) + p\beta\breve{g}_{\epsilon_n}(\tilde{u}^2_{\epsilon_n})
\breve{G}_{\epsilon_n}(\tilde{u}^1_{\epsilon_n}), &\text{in}\,\,\mathbb{R}^N.
  \end{array}
\right.
\end{align}
Testing \eqref{eq4.6} against with $(\tilde{u}^1_n,0)$, by \eqref{eq4.17}, we have
$$
\limsup_{n\to\infty}\frac{1}{\epsilon_n^N}\int_{\mathbb{R}^N}\epsilon_n^2|\nabla u^1_{\epsilon_n}|^2 + V_1(x)|u^1_{\epsilon_n}|^2dx\leq C \limsup_{n\to\infty}\|u^1_{\epsilon_n}\|^{2p - 2}_{L^{\infty}(\Lambda)} = 0,
$$
which implies
\begin{equation}\label{addeq3.1}
  \lim_{n\to\infty}\frac{\widetilde{\mathcal{C}}_{\epsilon_n}}{\epsilon^N_n}= \lim_{n\to\infty}\frac{1}{2}\int_{\R^N}|\nabla \tilde{u}^2_{n}|^2 + V^2_n(x)|\tilde{u}^2_{n}|^2 - \int_{R^N}\widehat{{G}}_{n}(\tilde{u}^2_n).
\end{equation}

For every $R > 0$, it holds that
\begin{align}\label{eq4.4}
\nonumber&\quad\lim_{n\to \infty}\int_{B_R}|\nabla \tilde{u}^i_{\epsilon_n}|^2 + V_{i}(x^2_*)|\tilde{u}^i_{\epsilon_n}|^2dx\\
&\leq \int_{B_R}|\nabla \tilde{u}^i_{\epsilon_n}|^2 + V^i_{\epsilon_n}(x)|\tilde{u}^i_{\epsilon_n}|^2dx\\
\nonumber& = \frac{1}{\epsilon_n^N}\int_{B_{\epsilon_n R}(x^2_{\epsilon_n})}\epsilon_n^2|\nabla u^i_{\epsilon_n}|^2 + V_i(x)|u^i_{\epsilon_n}|^2dx\\
&<\infty,
\end{align}
where we assume that $\lim_{n\to\infty}x^2_{\epsilon_n} = x^2_*$.
Then by diagonal argument, there exists $u^i_*\in H^1_{loc}(\mathbb{R}^N)$ such that $\tilde{u}^i_{\epsilon_n}\rightharpoonup u^i_*$ weakly in $H^1(B_R)$ as $\epsilon_n\to 0$. Since
$$
\int_{B_R}|\nabla u^i_*|^2 + V_{i}(x^2_*)|u^i_*|^2
\leq \liminf_{\epsilon_n\to 0}\frac{1}{\epsilon_n^N}\int_{\mathbb{R}^N}\epsilon_n^2|\nabla u^i_{\epsilon_n}|^2 + V_i(x)|u^i_{\epsilon_n}|^2dx < \infty,
$$
we have $u^i_*\in H^1(\mathbb{R}^N)$.
Especially, we have
\begin{equation}\label{eq4.5}
\|u^1_*\|^2_{H^1(\mathbb{R}^N)}
\leq \liminf_{\epsilon_n\to 0}\frac{1}{\epsilon_n^N}\int_{\mathbb{R}^N}\epsilon_n^2|\nabla u^i_{\epsilon_n}|^2 + V_i(x)|u^i_{\epsilon_n}|^2dx\leq C \liminf_{\epsilon_n\to 0}\|u^1_{\epsilon_n}\|^{2p - 2}_{L^{\infty}(\Lambda)} = 0,
\end{equation}
which implies that $u^1_* = 0$.

Denote $\Lambda^2_{\epsilon_n}:=\{x\in\R^N:\varepsilon_n x + x^2_{\epsilon_n}\in \Lambda\}$. Since $\Lambda$ is smooth, we have $\Lambda^2_{\epsilon_n}\to \Lambda^2_*\in\{\emptyset,H,\R^N\}$ as $n\to+\infty$, where $H$ is a half space. Following, by Cauchy inequality and the definition of $P_{\epsilon}$ in \eqref{eq3.1}, for every $\varphi\in C^{\infty}_c(\mathbb{R}^N)$, we have
\begin{align*}
&\quad \lim_{n\to\infty}p\beta\int_{\mathbb{R}^N}\varphi\breve{g}_{\epsilon_n}
(\tilde{u}^2_{\epsilon_n})
\breve{G}_{\epsilon_n}(\tilde{u}^1_{\epsilon_n}) = \int_{\R^N}\beta\varphi\chi_{\Lambda^2_*}(u^1_*)^{p}(u^2_*)^{p - 1}
\end{align*}
Hence, by a similar proof, we conclude that $w_* = (u^1_*,u^2_*)$ satisfies
\begin{equation}\label{eq4.7}
  \left\{
  \begin{array}{ll}
   -\Delta {u}^{1}_* + V_1(x^2_*){u}^{1}_*= \chi_{\Lambda^2_*}({u}^{1}_*)^{2p - 1} + \beta \chi_{\Lambda^2_*}({u}^{1}_*)^{p - 1}({u}^{2}_*)^p, & x\in \mathbb{R}^N,\vspace{0.12cm}\\
   -\Delta {u}^{2}_* + V_2(x^2_*){u}^{2}_*= \chi_{\Lambda^2_*}({u}^{2}_*)^{2p - 1}  + \beta \chi_{\Lambda^2_*}({u}^{2}_*)^{p - 1}({u}^{1}_*)^p, & x\in \mathbb{R}^N.
  \end{array}
\right.
\end{equation}
But since $u^1_* = 0$, system \eqref{eq4.7} is
\begin{equation}\label{adeq1}
-\Delta u^2_* + V_2(x^2_*)u^2_* = \chi_{\Lambda^2_*}(u^2_*)^{2p - 1}\,\,\text{in}\,\,\mathbb{R}^N.
\end{equation}
By Lemma \ref{le4.2} and the regularity argument in \cite{GT} we conclude that $u^2_*\not\equiv 0$. Hence $\Lambda^2_* = \mathbb{R}^N$, i.e.,
\begin{equation}\label{eq4.8}
  -\Delta u^2_* + V_2(x^2_*)u^2_* = (u^2_*)^{2p - 1}\,\,\text{in}\,\,\mathbb{R}^N.
\end{equation}

Now, if
\begin{equation*}\label{}
  \lim\limits_{{R\to \infty}\atop{{n\to\infty}}}\|u^2_{\epsilon_n}\|_{L^{\infty}(U\backslash B_{\epsilon_n R}(x^2_{\epsilon_n}))} =  0,
\end{equation*}
then letting $\eta_R$ be a smooth function with $0\leq \eta_R \leq 1$, $\eta_R = 0$ in $B_{\frac{R}{2}}$ and $\eta_R \equiv 1$ on $B^c_R$, testing \eqref{eq4.6} against with $(0,\eta_R \tilde{u}^{2}_n)$, we find
\begin{eqnarray*}
\begin{split}
&\quad \lim_{n\to\infty}\int_{\mathbb{R}^N\backslash B_R}|\nabla \tilde{u}^{2}_n|^2 + V^2_n(x)|\tilde{u}^{2}_n|^2\\
&\leq \lim_{n\to\infty}\Big(\int_{\mathbb{R}^N}|\nabla \eta_R||\tilde{u}^{2}_n||\nabla \tilde{u}^{2}_n| + \int_{\Lambda^2_n\backslash B_{\frac{R}{2}}}|\tilde{u}^{2}_n|^{2p} + \int_{\mathbb{R}^N\backslash\Lambda^2_{n}}P_{\epsilon_n}(\epsilon_n x + x^2_{\epsilon_n})|\tilde{u}^{2}_n|^{2}\Big)\\
&\leq \lim_{n\to\infty}\Big(\int_{\mathbb{R}^N}|\nabla \eta_R||\tilde{u}^{2}_n||\nabla \tilde{u}^{2}_n| + \|\tilde{u}^{2}_n\|^{2p - 2}_{L^{\infty}\big(U^2_{\epsilon_n}\backslash B_{\frac{R}{2}}\big)}\int_{\Lambda^2_n\backslash B_{\frac{R}{2}}}|\tilde{u}^{2}_n|^{2}\Big)\\
&\qquad + \lim_{n\to\infty}\Big(\int_{\mathbb{R}^N\backslash\Lambda^2_{n}}P_{\epsilon_n}(\epsilon_n x + x^2_{\epsilon_n})|\tilde{u}^{2}_n|^{2}\Big)\\
&= o_R(1).
\end{split}
\end{eqnarray*}
Following, by \eqref{eq4.6}, \eqref{eq4.8} and the uniqueness of positive solutions (in $H^1$) of $-\Delta u + u = u^{2p-1}$, we have
\begin{equation*}\label{addeq3.1}
  \lim_{n\to\infty}\frac{\widetilde{\mathcal{C}}_{\epsilon_n}}{\epsilon^N_n}= \Big(\frac{1}{2} - \frac{1}{2 p}\Big)\int_{\R^N}|\nabla u^2_*|^2 + V_2(x^2_*)|u^2_*|^2 + o_R(1) = \mathcal{C}_{V_2(x^2_*),0} + o_R(1).
\end{equation*}
Hence, letting $R\to\infty$, we have
\begin{equation*}\label{addeq3.1}
  \lim_{n\to\infty}\frac{\widetilde{\mathcal{C}}_{\epsilon_n}}{\epsilon^N_n}= \Big(\frac{1}{2} - \frac{1}{2 p}\Big)\int_{\R^N}|\nabla u^2_*|^2 + V_2(x^2_*)|u^2_*|^2 + o_R(1) = \mathcal{C}_{V_2(x^2_*),0}.
\end{equation*}

Proceeding as the proof above, by the similar blow-up analysis in Proposition 3.4 of \cite{28}, if
\begin{equation*}\label{}
  \lim\limits_{{R\to \infty}\atop{{n\to\infty}}}\|u^2_{\epsilon_n}\|_{L^{\infty}(U\backslash B_{\epsilon_n R}(x^2_{\epsilon_n}))} >  0,
\end{equation*}
then we can get a family of points $\{\hat{x}^2_{\epsilon_n}:0<\epsilon_n<\epsilon_0\}$ with $\hat{x}^2_{\epsilon_n}\to \hat{x}^2_*\in\overline{\Lambda}$. Moreover,
there holds
\begin{eqnarray*}\label{eq4.26}
\begin{split}
  \liminf_{n\to\infty}
\frac{\widetilde{\mathcal{C}}_{\epsilon_n}}{\epsilon^N_n}
& \geq \mathcal{C}_{V_2(x^2_*),0} + \mathcal{C}_{V_2(\hat{x}^2_*),0}.
\end{split}
\end{eqnarray*}
If
\begin{equation*}\label{}
  \lim\limits_{{R\to \infty}\atop{{n\to\infty}}}\|u^2_{\epsilon_n}\|_{L^{\infty}(U\backslash \bigcup_{i = 1}^2 B_{\epsilon_n R}(x^i_{\epsilon_n}))} =  0,
\end{equation*}
then proceeding as the proof above, we conclude that
\begin{eqnarray*}\label{eq4.26}
\begin{split}
  \limsup_{n\to\infty}
\frac{\widetilde{\mathcal{C}}_{\epsilon_n}}{\epsilon^N_n}
& \leq \mathcal{C}_{V_2(x^2_*),0} + \mathcal{C}_{V_2(\hat{x}^2_*),0},
\end{split}
\end{eqnarray*}
which implies that
\begin{eqnarray*}\label{eq4.26}
\begin{split}
  \limsup_{n\to\infty}
\frac{\widetilde{\mathcal{C}}_{\epsilon_n}}{\epsilon^N_n}
& = \mathcal{C}_{V_2(\hat{x}^2_*),0} + \mathcal{C}_{V_2(x^2_*),0}.
\end{split}
\end{eqnarray*}
But, since $\limsup_{\epsilon\to 0}\frac{{\mathcal{C}}_{\epsilon}}{\epsilon^N} < +\infty$, the above steps will stop at some positive integer $k\in\mathbb{N}$. We then complete the proof.

Next, we prove the concentration phenomenon of $w_{\epsilon}$. The proof is different with respect to the two cases $\omega=1$ and $\omega<1$. We first give the proof of the case $\omega =1$.

\begin{lemma}\label{le4.4}
Let $\rho > 0$,  $p\geq 2$ and $\omega = 1$. If
\begin{equation}
0<\beta < \beta_{1,p},
\end{equation}
then there exist three families of points $\{x^i_{\epsilon}\}$, $i = 1,2$ and $\{{x}_{\epsilon}\}\subset\Lambda$ such that
\begin{eqnarray*}
  &&(i)\ \liminf_{\epsilon\to 0}\|u^i_{\epsilon}\|_{L^{\infty}(B_{\epsilon\rho}(x^i_{\epsilon}))}> 0;\\
  &&(ii)\ \limsup_{\epsilon\to 0}\frac{|x^1_{\epsilon} - x^2_{\epsilon}| + |x^1_{\epsilon} - x_{\epsilon}|}{\epsilon} < +\infty;\\
  &&(iii)\lim_{\epsilon\to 0}dist\ (x^i_{\epsilon},\mathcal{M}) = 0;\\
  &&(iv)\ \lim\limits_{{R\to \infty}\atop{{\epsilon\to 0}}}\|u^i_{\epsilon}\|_{L^{\infty}(U\backslash B_{\epsilon R}(x^i_{\epsilon}))} =  0;\\
  &&(v)\ \liminf_{\epsilon\to 0}\|u^1_{\epsilon} + u^2_{\epsilon}\|_{L^{\infty}(B_{\epsilon\rho}(x_{\epsilon}))}> 0;\\
  &&(vi)\ \lim\limits_{{R\to \infty}\atop{{\epsilon\to 0}}}\|u^1_{\epsilon} + u^2_{\epsilon}\|_{L^{\infty}(U\backslash B_{\epsilon R}(x_{\epsilon}))} =  0.\\
\end{eqnarray*}
where $\beta_{\omega,p}$ is given in \eqref{CCeq3.9}.
\end{lemma}

\begin{proof}
$(i)$ follows easily by Lemma \ref{Cle3.8}. $(v)$ follows easily by the strong assumption \eqref{eq3.1}.

\textsf{Proof of $(ii)$}. We argue by contradiction. Suppose to the contrary that there exists a subsequence $\{\epsilon_n:n\in\mathbb{N}\}\subset(0,\epsilon_0)$ with $\lim_{n\to\infty}\epsilon_n = 0$, such that
$$
\limsup_{n\to +\infty}\frac{|x^1_{\epsilon_n} - x^2_{\epsilon_n}|}{\epsilon_n}= +\infty.
$$

For each $j =1,2$, denote $\tilde{w}^j_n(x) = \big(u^1_{\epsilon_n}(\epsilon_n x + x^j_{\epsilon_n}),u^2_{\epsilon_n}(\epsilon_n x + x^j_{\epsilon_n})\big)=(\tilde{u}^{1j}_{\epsilon_n},\tilde{u}^{2j}_{\epsilon_n})$, $V^{ij}_{\epsilon_n}(x) = V_i(\epsilon_n x + x^j_{\epsilon_n})$,
$\hat{g}^j_{\epsilon_n}(\cdot) = g_{\epsilon_n}(\epsilon_nx + x^j_{\epsilon
_n},\cdot)$, $\widehat{G}^j_{\epsilon_n}(\cdot) = G_{\epsilon_n}(\epsilon_nx + x^j_{\epsilon_n},\cdot)$, $\breve{g}^j_{\epsilon_n}(\cdot) = \tilde{g}_{\epsilon_n}(\epsilon_nx + x^j_{\epsilon_n},\cdot)$ and $\breve{G}^j_{\epsilon_n}(\cdot) = \widetilde{G}_{\epsilon_n}(\epsilon_nx + x^j_{\epsilon_n},\cdot)$. Then similar to \eqref{eq4.6}, $\tilde{w}^j_{\epsilon_n}$ satisfies:
\begin{align}\label{eq4.11}
\left\{
  \begin{array}{ll}
    -\Delta \tilde{u}^{1j}_{\epsilon_n} + V^{1j}_{\epsilon_n}(x)\tilde{u}^{1j}_{\epsilon_n} = \hat{g}^j_{\epsilon_n}(\tilde{u}^{1j}_{\epsilon_n}) + p\beta\breve{g}^j_{\epsilon_n}(\tilde{u}^{1j}_{\epsilon_n})
\breve{G}^j_{\epsilon_n}(\tilde{u}^{2j}_{\epsilon_n}), &  \text{in}\,\,\mathbb{R}^N,\vspace{0.12cm}\\
    -\Delta \tilde{u}^{2j}_{\epsilon_n} + V^{2j}_{\epsilon_n}(x)\tilde{u}^{2j}_{\epsilon_n} = \hat{g}^j_{\epsilon_n}(\tilde{u}^{2j}_{\epsilon_n}) + p\beta\breve{g}^j_{\epsilon_n}(\tilde{u}^{2j}_{\epsilon_n})
\breve{G}^j_{\epsilon_n}(\tilde{u}^{1j}_{\epsilon_n}), &\text{in}\,\,\mathbb{R}^N.
  \end{array}
\right.
\end{align}
Also similar to the proof of Lemma \ref{le4.4}, there exists $\tilde{w}^j_* = (\tilde{u}^{1j}_*,\tilde{u}^{2j}_*)\in\mathcal{H}$ such that
$\tilde{w}^j_{\epsilon_n}\to \tilde{w}^j_*$ weakly in $\mathcal{H}_{loc}$. Moreover,
similar to \eqref{eq4.7}, we conclude that $\tilde{w}^j_*$ satisfies
\begin{equation*}\label{eq4.12}
  \left\{
  \begin{array}{ll}
   -\Delta \tilde{u}^{1j}_* + V_1(x^j_*)\tilde{u}^{1j}_*= \chi_{\Lambda^j_*}(\tilde{u}^{1j}_*)^{2p - 1} + \beta \chi_{\Lambda^j_*}(\tilde{u}^{1j}_*)^{p - 1}(\tilde{u}^{2j}_*)^p, & x\in \mathbb{R}^N,\vspace{0.12cm}\\
   -\Delta \tilde{u}^{2j}_* + V_2(x^j_*)\tilde{u}^{2j}_*= \chi_{\Lambda^j_*}(\tilde{u}^{2j}_*)^{2p - 1}  + \beta \chi_{\Lambda^j_*}(\tilde{u}^{2j}_*)^{p - 1}(\tilde{u}^{1j}_*)^p, & x\in \mathbb{R}^N,
  \end{array}
\right.
\end{equation*}
where we assume that $\Lambda^j_*$ and $x^j_*$ are the limits of $\{x\in\mathbb{R}^N:\epsilon x + x^j_{\epsilon_n}\in\Lambda\}$ and $x^j_{\epsilon_n}$ respectively. Noting that  $\tilde{u}^{11}_*$ and $\tilde{u}^{22}_*$ are nonstandard, hence by Lemma \ref{le4.2}, we conclude that $\Lambda^j_* = \mathbb{R}^N$, i.e.,
\begin{equation}\label{eq4.13}
  \left\{
  \begin{array}{ll}
   -\Delta \tilde{u}^{1j}_* + V_1(x^j_*)\tilde{u}^{1j}_*= \chi_{\mathbb{R}^N}(\tilde{u}^{1j}_*)^{2p - 1} + \beta (\tilde{u}^{1j}_*)^{p - 1}(\tilde{u}^{2j}_*)^p, & x\in \mathbb{R}^N,\vspace{0.12cm}\\
   -\Delta \tilde{u}^{2j}_* + V_2(x^j_*)\tilde{u}^{2j}_*= \chi_{\mathbb{R}^N}(\tilde{u}^{2j}_*)^{2p - 1}  + \beta (\tilde{u}^{2j}_*)^{p - 1}(\tilde{u}^{1j}_*)^p, & x\in \mathbb{R}^N.
  \end{array}
\right.
\end{equation}

Now for every $R > 0$, by Sobolev embedding theorem, we have
\begin{align}\label{eq4.14}
\nonumber &\quad \liminf_{n\to\infty}\frac{1}{\epsilon_n^N}
\Big(\frac{1}{2}\int_{ B_{\epsilon_n R}(x^1_{\epsilon_n})\cup B_{\epsilon_n R}(x^2_{\epsilon_n})}
\epsilon_n^2|\nabla u^1_{\epsilon_n}|^2 + V_1(x)|u^1_{\epsilon_n}|^2 - \int_{ B_{\epsilon_n R}(x^1_{\epsilon_n})\cup B_{\epsilon_n R}(x^2_{\epsilon_n})}\mathfrak{G}_{\epsilon_n}(u^1_{\epsilon_n})\\
\nonumber&\,\,\,\,\,\,\,\,\,\,\,\,\,\,\,\,\,\,\,\,\,\,\,\,\,\,\,\, + \frac{1}{2}\int_{ B_{\epsilon_n R}(x^1_{\epsilon_n})\cup B_{\epsilon_n R}(x^2_{\epsilon_n})}
\epsilon_n^2|\nabla u^2_{\epsilon_n}|^2 + V_1(x)|u^2_{\epsilon_n}|^2 - \int_{ B_{\epsilon_n R}(x^1_{\epsilon_n})\cup B_{\epsilon_n R}(x^2_{\epsilon_n})}\mathfrak{G}_{\epsilon_n}(u^2_{\epsilon_n})\\
\nonumber&\,\,\,\,\,\,\,\,\,\,\,\,\,\,\,\,\,\,\,\,\,\,\,\,\,\,\,\, -
p\beta\int_{ B_{\epsilon_n R}(x^1_{\epsilon_n})\cup B_{\epsilon_n R}(x^2_{\epsilon_n})}
\widehat{\mathfrak{G}}_{\epsilon_n}(u^1_{\epsilon_n})
\widehat{\mathfrak{G}}_{\epsilon_n}(u^2_{\epsilon_n})\Big)\\
& = \liminf_{n\to\infty}\bigg[\sum_{j = 1}^2
\Big(\frac{1}{2}\int_{B_{R}}
|\nabla \tilde{u}^{1j}_{\epsilon_n}|^2 + V^{1j}_{\epsilon_n}(x)|\tilde{u}^{1j}_{\epsilon_n}|^2 - \int_{ B_{R}}\widehat{\mathfrak{G}}^{1j}_{\epsilon_n}(\tilde{u}^{1j}_{\epsilon_n})\\
\nonumber&\,\,\,\,\,\,\,\,\,\,\,\,\,\,\,\,\,\,\,\,\,\,\,\,\,\,\,\, + \frac{1}{2}\int_{ B_{R}}
|\nabla \tilde{u}^{2j}_{\epsilon_n}|^2 + V^{2j}_{\epsilon_n}(x)|u^2_{\epsilon_n}|^2 - \int_{ B_{R}}\widehat{\mathfrak{G}}^2_{\epsilon_n}
(\tilde{u}^{2j}_{\epsilon_n})\Big)\\
\nonumber&\,\,\,\,\,\,\,\,\,\,\,\,\,\,\,\,\,\,\,\,\,\,\,\,\,\,\,\, -
\int_{B_{R}}
p\beta{\breve{\mathfrak{G}}^{1}_{\epsilon_n}(\tilde{u}^{11}_{\epsilon_n})}
{\breve{\mathfrak{G}}^1_{\epsilon_n}(\tilde{u}^{21}_{\epsilon_n})}
 - p\beta\int_{B_{R}}
{\breve{\mathfrak{G}}^{2}_{\epsilon_n}(\tilde{u}^{12}_{\epsilon_n})}
{\breve{\mathfrak{G}}^2_{\epsilon_n}(\tilde{u}^{22}_{\epsilon_n})}\bigg]\\
\nonumber&\geq \sum_{j = 1}^2J_{V_1(x^j_*),V_2(x^j_*),\beta}(\tilde{w}^j_*) + o_R(1).
\end{align}
Let  $\eta\in C^{\infty}(\mathbb{R}^N)$ be the cut-off function satisfying $0\leq \eta_R \leq 1$, $\eta_R = 0$ in $B_{\frac{R}{2}}$ and $\eta_R \equiv 1$ on $B^c_R$, testing the penalized equation \eqref{Ceq3.3} against with $\eta\Big(\frac{x - x^1_{\epsilon_n}}{\epsilon_n}\Big)\eta\Big(\frac{x - x^2_{\epsilon_n}}{\epsilon_n}\Big)w_{\epsilon_n} : = \hat{\eta}w_{\epsilon_n} $, by the similar blow-up analysis in Lemma \ref{le4.4}, we have

\begin{align}\label{eq4.15}
\nonumber & \liminf_{n\to\infty}\frac{1}{\epsilon^N_n}\bigg( \frac{1}{2}\int_{\mathbb{R}^N\backslash B_{\epsilon_n R}(x^1_{\epsilon_n})\cup B_{\epsilon_n R}(x^2_{\epsilon_n})}
\epsilon_n^2|\nabla u^1_{\epsilon_n}|^2 + V_1(x)|u^1_{\epsilon_n}|^2\\
\nonumber&\quad - \int_{\mathbb{R}^N\backslash B_{\epsilon_n R}(x^1_{\epsilon_n})\cup B_{\epsilon_n R}(x^2_{\epsilon_n})}\mathfrak{G}_{\epsilon_n}(u^1_{\epsilon_n})\\
\nonumber&\quad + \frac{1}{2}\int_{\mathbb{R}^N\backslash B_{\epsilon_n R}(x^1_{\epsilon_n})\cup B_{\epsilon_n R}(x^2_{\epsilon_n})}
\epsilon_n^2|\nabla u^2_{\epsilon_n}|^2 + V_2(x)|u^2_{\epsilon_n}|^2\\
&\quad - \int_{\mathbb{R}^N\backslash B_{\epsilon_n R}(x^1_{\epsilon_n})\cup B_{\epsilon_n R}(x^2_{\epsilon_n})}\mathfrak{G}_{\epsilon_n}(u^2_{\epsilon_n})\\
\nonumber&\quad-
{p\beta}\int_{\mathbb{R}^N\backslash B_{\epsilon_n R}(x^1_{\epsilon_n})\cup B_{\epsilon_n R}(x^2_{\epsilon_n})}
\widetilde{\mathfrak{G}}_{\epsilon_n}(u^1_{\epsilon_n})
\widetilde{\mathfrak{G}}_{\epsilon_n}(u^2_{\epsilon_n})\bigg)\\
\nonumber& \geq \frac{1}{2}\liminf_{n\to\infty}\bigg(\sum_{j = 1}^2\Big(\int_{\mathbb{R}^N}\nabla{\eta}(u^{1j}_{\epsilon_n}\nabla u^{1j}_{\epsilon_n} + u^{2j}_{\epsilon_n}\nabla u^{2j}_{\epsilon_n})\\
\nonumber&\quad  - \int_{\mathbb{R}^N}(1 - {\eta})
\big(\mathfrak{g}^j_{\epsilon_n}
(u^{1j}_{\epsilon_n})u^{1j}_{\epsilon_n} + \mathfrak{g}^j_{\epsilon_n}
(u^{2j}_{\epsilon_n})u^{2j}_{\epsilon_n}\big)\Big)\\
\nonumber&\quad -
{p\beta}\int_{\mathbb{R}^N}(1 - {\eta})
\big(\tilde{\mathfrak{g}}^1_{\epsilon_n}(u^{11}_{\epsilon_n})
\widetilde{G}^1_{\epsilon_n}(u^{21}_{\epsilon_n})u^{11}_{\epsilon_n} + \tilde{\mathfrak{g}}^2_{\epsilon_n}(u^{22}_{\epsilon_n})
\widetilde{G}^2_{\epsilon_n}(u^{11}_{\epsilon_n})u^{22}_{\epsilon_n}\bigg)\\
\nonumber& \geq o_R(1).
\end{align}
Hence, combining with \eqref{eq4.14} and \eqref{eq4.15}, letting $R\to\infty$, we conclude that
\begin{equation}\label{eq4.16}
  \liminf_{n\to\infty}\frac{\mathcal{C}_{\epsilon_n}}{\epsilon_n}\geq \sum_{j = 1}^2J_{V_1(x^j_*),V_2(x^j_*),\beta}(\tilde{w}^j_*).
\end{equation}
But, by Theorem \ref{th2.2}, it holds
\begin{equation}\label{eq4.33}
\liminf_{\epsilon\to 0}\frac{{\mathcal{C}}_{\epsilon}}{\epsilon^N}
>2\mathcal{C}_{\textsf{m}_,0},
\end{equation}
which is a contradiction to \eqref{eq4.11}. This proves $(ii)$.

\textsf{Proof of $(iv)-(vi)$.} Arguing by contradiction, we get two systems like \eqref{eq4.13}. Then by Theorem \ref{th2.2}, we get \eqref{eq4.33}, which is a contradiction.

\textsf{Proof of $(iii)$.} By $(ii)$, supposing that $\lim_{\epsilon\to 0}x^1_{\epsilon}= x_* = \lim_{\epsilon\to 0}x_{\epsilon}$. Then by Theorem \ref{th2.2} and the same blow-up analysis above, we have
\begin{equation*}\label{}
2\mathcal{C}_{\textsf{m},\beta}\geq \liminf_{\epsilon\to 0}\frac{{\mathcal{C}}_{\epsilon}}{\epsilon^N}\geq  \mathcal{C}_{V_1(x_*),\beta} + \mathcal{C}_{V_2(x_*),\beta},
\end{equation*}
from which we deduce that
\begin{equation*}
V_i(x_*) = \textsf{m}.
\end{equation*}
Then $x_*\in\mathcal{M}$. This completes the proof.
\end{proof}

%

\textbf{L}et $p\geq 2$. Lemma \ref{Cle3.8} and the classical regularity argument(\cite{GT}) imply that there exists $x^i_{\omega,\epsilon}\in\overline{\Lambda}$, $i = 1,2$, $x_{\omega,\epsilon}$ such that $$
{u}^i_{\epsilon}({x}^i_{\omega,\epsilon}) = \max_{\overline{\Lambda}} {u}^i_{\epsilon},\ (u^1_{\epsilon} + u^2_{\epsilon})(x_{\omega,\epsilon}) = \max_{\overline{\Lambda}}(u^1_{\epsilon} + u^2_{\epsilon}).
$$
Since in Lemma \ref{le4.4} we have proved that when $\omega = 1$, it must hold
\begin{equation}\label{eq4.40}
  \lim_{\epsilon\to 0} d({x}^i_{\omega,\epsilon},\mathcal{M}) = 0,\ i = 1,2,
\end{equation}
which implies ${u}^i_{\epsilon}$, $i= 1,2$ must concentrate synchronously at the common local minimum of $V_1$ and $V_2$.
However, it is very difficult to
prove such concentration phenomenon for the case $\omega\neq 1$. On one hand, {there is no monotonicity for higher energy}. On the other hand, since $\omega \neq 1$, we can just estimate $\mathcal{C}^*_{\textsf{m}_1,\textsf{m}_2,\beta}$ in \eqref{eq3.9} rather than compute it as a precise number. To handle this big obstacle, we first show that ${x}^i_{\omega,\epsilon}$, $i = 1,2$, will be far away from $\partial\Lambda$ if we let $\beta>0$ be small enough:
\begin{lemma}\label{le4.5}
There exists a positive constant $\breve{\beta}>0$ such that if $0<\beta<\breve{\beta}$, then the properties $(i)$, $(ii)$, $(iv)$, $(v)$ and $(vi)$ in Lemma \ref{le4.4} hold  and

$(iii)'$ $\liminf_{\epsilon\to 0} d({x}_{\omega,\epsilon},\partial\Lambda)\big) > 0.$
\end{lemma}

\begin{proof}
Properties $(i)$ and $(v)$ are from the proof of Lemma \ref{le4.4}. Suppose that $(ii)$ is not true. Then by the same proof in Lemma \ref{le4.4}, $\liminf_{\epsilon\to 0}\frac{\mathcal{C}_{\epsilon}}{\epsilon^N}$ is bounded below by one of the following three numbers
\begin{eqnarray}\label{eq4.41}
\begin{split}
(1)\ 2\big(\mathcal{C}_{\textsf{m}_1,\beta} + \mathcal{C}_{\textsf{m}_2,\beta}\big),\ (2)\ \mathcal{C}_{\textsf{m}_1,0} + \mathcal{C}_{\textsf{m}_1,\beta} + \mathcal{C}_{\textsf{m}_2,\beta},\ \text{and}\ (3)\ \mathcal{C}_{\textsf{m}_1,0} + \mathcal{C}_{\textsf{m}_2,0},\\
\end{split}
\end{eqnarray}
which is a contradiction to \eqref{eq3.9} if we let
$$
0<\breve{\beta}\le \beta_{\omega,p},
$$
where $\beta_{\omega,p}$ is the constant in \eqref{CCeq3.9}.

By $(ii)$ we can prove easily that if $(iv)$ and $(vii)$ are not true, then $\liminf_{\epsilon\to 0}\frac{\mathcal{C}_{\epsilon}}{\epsilon^N}$ is bounded below by one of those numbers in \eqref{eq4.41}, which is a contradiction if $0<\beta<\beta_{\omega,p}$.

Suppose that $(iii)'$ is not true, i.e.,
$$
 \lim_{\epsilon\to 0} {x}_{\omega,\epsilon} = {x}^*\in\partial\Lambda.
$$
Then by Theorem \ref{th2.2} and the same blow-up analysis in Lemma \ref{le4.4}, it holds
\begin{equation}\label{eq4.42}
  \liminf_{\epsilon\to 0}\frac{\mathcal{C}_{\epsilon}}{\epsilon^N}\geq \mathcal{C}_{V_1({x}^*),\beta} + \mathcal{C}_{V_2({x}^*),\beta}.
\end{equation}
Since ${x}^*\in\partial\Lambda$, by \eqref{AAeq1.6} there exists a positive constant $c_1>0$ which is independent of ${x}^*$ such that
\begin{equation}\label{neweq11}
V_1({x}^*) + V_2({x}^*)>\textsf{m}_1 + \textsf{m}_2 + c_1,
\end{equation}
from which we can let $\tilde{{\beta}}>0$ be the unique constant such that
\begin{equation}\label{neweq6}
  \mathcal{C}_{V_1({x}^*),\tilde{\beta}} + \mathcal{C}_{V_2({x}^*),\tilde{\beta}}
=\mathcal{C}_{\textsf{m}_1,0} + \mathcal{C}_{\textsf{m}_2,0}.
\end{equation}
Moreover, by \eqref{neweq11}, we conclude that there exists a positive constant $c_2>0$ which is only related to $c_1$ such that
\begin{equation*}
\tilde{\beta}>c_2.
\end{equation*}
Then we get a contradiction from \eqref{eq4.42} and \eqref{eq3.9} once we let
\begin{equation}\label{eq4.44}
  \beta < \breve{\beta}: = \min\{\beta_{\omega,p},\tilde{\beta}\}.
\end{equation}
This completes the proof.
\end{proof}

\textbf{N}ow, we define the constant $\bar{\beta}$ in Theorem \ref{th1.2} as
\begin{equation}\label{neweq12}
  \bar{\beta} = \frac{1}{(1 + \kappa)}\min\{\breve{\beta},\tilde{\beta}_{\textsf{m}_1,\textsf{m}_2,p}\},
\end{equation}
where $\kappa>0$ is a small parameter, $\tilde{\beta}_{\textsf{m}_1,\textsf{m}_2,p}$ is the constant in Theorem \ref{important}.

\vspace{0.5cm}

One can find by the construction of the penalized function $P_{\epsilon}$(\eqref{eq3.1}) in the coming Section 5 that property $(iii)'$ is sufficient to prove that the penalized solution $w_{\epsilon}$ in Lemma \ref{le3.8} solves the original problem \eqref{eq1.1}. A natural question is whether it holds
\begin{equation}\label{eq4.45}
  {x}_{\omega,*}\in\mathcal{M}?
\end{equation}
If \eqref{eq4.45} is true, then the location of concentration points is at $\mathcal{M}$, the common local minimum of $V_1$ and $V_2$.
  It is very interesting that we answer positively to \eqref{eq4.45} by using local Pohozaev identities. We emphasize that to achieve this goal, we need some decay estimates for the solution ${w}_{\epsilon}$, which is obtained by the skillful construction of penalized function $P_{\epsilon}$(see \eqref{eq3.1}). For the continuity, we postpone the following decay estimates in Section \ref{s6}:
\begin{lemma}\label{LE4.6}
There hold
\begin{equation*}\label{Deq4.46}
\chi_{B_{\delta/\epsilon}}|{u}^1_{\epsilon}(\epsilon x + {x}_{\omega,\epsilon}) + {u}^2_{\epsilon}(\epsilon x + {x}_{\omega,\epsilon})|^2\ \ \text{is uniformly integrable}
\end{equation*}
and
\begin{eqnarray*}
\begin{split}
{u}^1_{\epsilon}(x) + {u}^2_{\epsilon}(x),\ \ |\nabla {u}^1_{\epsilon}(x)| + |\nabla {u}^2_{\epsilon}(x)|\leq \epsilon^No_{\epsilon}(1)\ \ \text{for all}\ x\in\partial B_{\delta}({x}_{\omega,\epsilon}),
\end{split}
\end{eqnarray*}
where $\delta>0$ is a small constant.
\end{lemma}

At the last of this section, we use Lemma \ref{LE4.6} and local Pohozaev identities to prove:
\begin{theorem}\label{th4.6}
Let $V_1,V_2\in C^1(\mathbb{R}^N)$ and
$$
\lim_{\epsilon\to 0}{x}_{\omega,\epsilon} = {x}_{\omega,*}.
$$
Then it holds
$$
\nabla V_1({x}_{\omega,*}) = \nabla V_2({x}_{\omega,*}) = 0\ \ \text{and}\ {x}_{\omega,*}\in\mathcal{M}
$$
(after necessary arrangement of $\Lambda$ when $\omega<1$).
\end{theorem}

\begin{proof}

The local Pohozaev identities are derived as follows. By testing system \eqref{eq1.1} against with $\big(\frac{\partial {u}^1_{\epsilon}}{\partial x_i}, \frac{\partial {u}^2_{\epsilon}}{\partial x_i}\big)$ and integrating on $B_{\delta}(x_{\omega,\epsilon})$, we find
\begin{eqnarray*}\label{Feq6.1}
\begin{split}
&\quad\int_{\partial B_{\delta}(x_{\omega,\epsilon})}\frac{\partial {u}^1_{\epsilon}}{\partial x_i}\frac{\partial {u}^1_{\epsilon}}{\partial \nu} dS + \int_{\partial B_{\delta}(x_{\omega,\epsilon})}\frac{\partial {u}^2_{\epsilon}}{\partial x_i}\frac{\partial {u}^2_{\epsilon}}{\partial \nu} dS - \int_{\partial B_{\delta}(x_{\omega,\epsilon})}\big(|\nabla {u}^1_{\epsilon}|^2 + |\nabla {u}^2_{\epsilon}|^2)\nu_i dS\\
&\quad + \frac{1}{2}\int_{\partial B_{\delta}(x_{\omega,\epsilon})} \big(V_1(x)|{u}^1_{\epsilon}|^2 + |{u}^2_{\epsilon}(x)|^2\big)\nu_i\\
&\quad+ \int_{\partial B_{\delta}(x_{\omega,\epsilon})}\big(|{u}^1_{\epsilon}|^{2p} + |{u}^1_{\epsilon}|^{2p} + \beta |{u}^1_{\epsilon}|^p|{u}^2_{\epsilon}|^p\big)\nu_idS  \\
& = \frac{1}{2}\int_{B_{\delta}(x_{\omega,\epsilon})}
\Big(|{u}^1_{\epsilon}|^2\frac{\partial V_1}{\partial x_i} + |{u}^2_{\epsilon}|^2\frac{\partial V_2}{\partial x_i}\Big),
\end{split}
\end{eqnarray*}
for every $i = 1,\ldots,N$, where $\delta > 0$ is a positive constant. Then by Lemma \ref{LE4.6}, we have
\begin{eqnarray*}\label{eq6.4}
\begin{split}
\epsilon^No_{\epsilon}(1) &= \int_{B_{\delta}(x_{\omega,\epsilon})}
\Big(|{u}^1_{\epsilon}|^2\frac{\partial V_1}{\partial x_i} + |{u}^2_{\epsilon}|^2\frac{\partial V_2}{\partial x_i}\Big)\\
&= \epsilon^N \int_{B_{\delta/\epsilon}}\big(|\tilde{u}^1_{\epsilon}|^2\nabla_iV_1(\epsilon x + {x}_{\omega,\epsilon}) + |\tilde{u}^2_{\epsilon}|^2\nabla_iV_2(\epsilon x + {x}_{\omega,\epsilon})\big) \\
& = \epsilon^N\Big(
C^i_1\epsilon^N\Big(\frac{\partial V_1(x)}{\partial x_i}\Big|_{x = x_{\omega,\epsilon}} + o_{\epsilon}(1)\Big) + C^i_2\epsilon^N\Big(\frac{\partial V_2(x)}{\partial x_i}\Big|_{x = x_{\omega,\epsilon}} + o_{\epsilon}(1)\Big),
\end{split}
\end{eqnarray*}
where $C^i_1,\ 1\leq i\leq N$ are positive constants. Then by Lemma \eqref{Deq4.46}, we have
\begin{equation}\label{eq4.50}
 C^i_1\nabla_iV_1({x}_{\omega,*}) + C^i_2\nabla_iV_1({x}_{\omega,*}) = 0.
\end{equation}

Finally, since $V_1,V_2$ are $C^1$ and $\mathcal{M}\neq \emptyset$ is a compact set, there exists a $\delta>0$, such that for every $x\in\mathcal{M}$ the functions $f^i_{j,x}:[-\delta,\delta]\to [0,+\infty)$, $1\leq i\leq N,\ j = 1,2$ defined as
\begin{equation*}\label{}
 f^i_{j,x}(t) = V_j(x + t\vec{e}_i)
\end{equation*}
satisfy
\begin{equation*}\label{}
  (f^i_{j,x})'(t)\ge 0\ \text{if}\ t\in[0,\delta]\ \ \text{and}\ \ (f^i_{j,x})'(t)\le 0\ \text{if}\ t\in[-\delta,0].
\end{equation*}
Thus, by continuity and compactness, there exists a smaller $\tilde{\delta}>0$ such that
\begin{equation*}\label{}
  \nabla_iV_1(x)\nabla_iV_2(x)\geq 0\ \ \text{for all}\ \ x\in(\mathcal{M})^{\tilde{\delta}}.
\end{equation*}
Hence, rearranging $\Lambda = (\mathcal{M})^{\tilde{\delta}}$, we have
\begin{equation*}\label{}
  \nabla_iV_1(x)\nabla_iV_2(x)\geq 0\ \ \text{for all}\ \ x\in\Lambda,
\end{equation*}
which and \eqref{eq4.50} imply
\begin{equation*}\label{}
 \nabla_iV_1({x}_{\omega,*})= 0 =\nabla_iV_2({x}_{\omega,*})\ \ \text{for all}\ 1\le i\le N.
\end{equation*}
Then we conclude that ${x}_{\omega,*}$ in $\mathcal{M}$ if $\Lambda$ is smaller again if necessary.
\end{proof}

\begin{remark}
The proof of Theorem \ref{th4.6} implies
 the coupling constant $\beta$ should be small if necessary. But if $V_2(x) = CV_1(x)$ with $C>1$ is a constant, the coupling constant $\beta$ can be large.
\end{remark}

\section{Back to the original problem}\label{s5}
In this section, we are going to prove that the penalized solution $w_{\epsilon}$ obtained in Lemma \ref{le3.8} solves the original problem \eqref{eq1.1}. What we need to do is to construct a suitable penalized function $P_{\epsilon}$ such that not only \eqref{eq3.1} is true, but also it holds
\begin{equation}\label{eq5.1}
(u^i_{\epsilon})^{2p - 2}_+(x)\leq P_{\epsilon}(x),\,\,i = 1,2.
\end{equation}
Noting that once \eqref{eq5.1} is true, combining with the result in Section \ref{s4}, we immediately have
$$
\lim\limits_{{R\to \infty}\atop{{\epsilon\to 0}}}\|u^i_{\epsilon}\|_{L^{\infty}(\mathbb{R}^N\backslash B_{\epsilon R}(x^i_{\epsilon}))} =  0,
$$
which means the concentration phenomenon of $w_{\epsilon}$.

We will use the comparison principle of the single equation \eqref{eq5.9} below to prove \eqref{eq5.1}. Firstly, we need to linearize the penalized system \eqref{Ceq3.3} outside small balls.

\begin{proposition}\label{pr4.1}
Let $\epsilon > 0$ be small enough, $\delta\in (0,1)$,  ${x}^i_{\epsilon,\omega}$$(i = 1,2)$ and ${x}_{\omega,\epsilon}$ be the points that are given by Lemma  \ref{le4.5}. Then there exists $R > 0$,
such that
\begin{equation}\label{eq5.2}
      -\epsilon^{2}\Delta v_{\epsilon} + (1 - \delta)V_{min}(x)v_{\epsilon}\leq P_{\epsilon} v_{\epsilon}\ \  \text{in}\ \mathbb{R}^N\backslash  B_{R\epsilon}(x_{\omega,\epsilon}),
\end{equation}
where $v_{\epsilon} := u^1_{\epsilon} + u^2_{\epsilon}$ and
\begin{equation}\label{keq5.2}
x_{\omega,\epsilon} := \frac{x^1_{\epsilon} + x^2_{\epsilon} + x_{\epsilon}}{3}
\end{equation}
 respectively.
\end{proposition}

\begin{proof}
By Lemma \ref{le4.4}, there exists $R > 0$ such that for all $x\in U\backslash  B_{R\epsilon}(x_{\omega,\epsilon})$,
\begin{align*}
 &\quad\mathfrak{g}_{\epsilon}(u^1_{\epsilon}) + p\tilde{\mathfrak{g}}_{\epsilon}(u^1_{\epsilon})\widetilde{\mathfrak{G}}_{\epsilon}(u^2_{\epsilon})
\leq \delta \chi_{\Lambda}(u^1_{\epsilon} + u^2_{\epsilon})
\end{align*}
and if $x\in \mathbb{R}^N\backslash\Lambda$,
\begin{align*}
 &\quad\mathfrak{g}_{\epsilon}(u^1_{\epsilon}) + p\tilde{\mathfrak{g}}_{\epsilon}(u^1_{\epsilon})\widetilde{\mathfrak{G}}_{\epsilon}(u^2_{\epsilon})
\leq P_{\epsilon}(u^1_{\epsilon} + u^2_{\epsilon}).
\end{align*}
Hence we have
$$
 -\epsilon^{2}\Delta u^1_{\epsilon} + V_1(x)u^1_{\epsilon}\leq \delta \chi_{\Lambda}(u^1_{\epsilon} + u^2_{\epsilon}) + P_{\epsilon}(u^1_{\epsilon} + u^2_{\epsilon}).
$$
Similarly, we have
$$
 -\epsilon^{2}\Delta u^2_{\epsilon} + V_2(x)u^2_{\epsilon}\leq \delta \chi_{\Lambda}(u^1_{\epsilon} + u^2_{\epsilon}) + P_{\epsilon}(u^1_{\epsilon} + u^2_{\epsilon}).
$$
Then we complete the proof.
\end{proof}

We are now in a position to construct penalized solutions for the linearized system in Proposition \ref{pr4.1}. By the penalized function, it is enough to check that the penalized solution $w_{\epsilon}$ solve the original problem \eqref{eq1.1} (see \eqref{eq5.9} below).
Moreover, the penalized function makes us obtain a good decay about the solution $w_{\epsilon}$ in Lemma \ref{le3.8}, which is necessary in verifying the assumptions in Lemma \ref{LE4.6}, see section \ref{s6} below.

Noting that by the classical bootstrap argument and nonnegativeness of $w_{\epsilon}$, we can conclude that
\begin{equation}\label{eq5.5}
  \limsup_{\epsilon\to 0}(\|u^1_{\epsilon}\|_{L^{\infty}(\Lambda)} + \|u^2_{\epsilon}\|_{L^{\infty}(\Lambda)})\leq \widetilde{C} < \infty.
\end{equation}

\begin{proposition}\label{pr5.3}
(Construction of barrier functions) Let $\{x_{\omega,\epsilon}\}\subset \Lambda$ be the family of points that are given in \eqref{keq5.2}. Assume that either \eqref{AAeq1.8} or \eqref{AAeq1.9} holds.
Then for sufficiently small $\epsilon > 0$, there exist $\overline{U}_{\epsilon}\in \mathcal{H}_{\epsilon}\cap C^{2}(\mathbb{R}^N)$ and $P_{\epsilon}$ satisfying the assumption \eqref{eq3.1}, such that $\overline{{U}}_{\epsilon} > 0$ satisfies
\begin{equation*}
  \left\{
    \begin{array}{ll}
     -\epsilon^2\Delta \overline{U}_{\epsilon} + (1 - \delta)V_{min}(x)\overline{U}_{\epsilon}\geq P_{\epsilon}\overline{U}_{\epsilon}, & \text{in}\ \mathbb{R}^N\backslash B_{R\epsilon}(x_{\omega,\epsilon}),\vspace{0.12cm} \\
      \overline{U}_{\epsilon}\geq \widetilde{C}, & \text{in}\  B_{R\epsilon}(x_{\omega,\epsilon}),
    \end{array}
  \right.
\end{equation*}
where $\widetilde{C}$ is the constant in \eqref{eq5.5}. Moreover, $\overline{{U}}^{2p - 2}_{\epsilon} < P_{\epsilon}$ in $\mathbb{R}^N\backslash \Lambda$.
\end{proposition}

\begin{proof}

Let
\begin{equation}\label{space}
r = \frac{1}{3}\min_{i = 1,2}\liminf_{\epsilon\to 0}dist (x_{\omega,\epsilon},\partial\Lambda)
\end{equation}
(this is reasonable because of the estimates $(iii)$ in Lemma \ref{le4.4} and $(iii)'$ in Lemma \ref{le4.5}). Define
\begin{equation}\label{eq5.7}
p^{\nu}_{\epsilon}(x) = \left\{
                       \begin{array}{ll}
                         \overline{C}\Big(1 + \frac{\nu(r - |x - x_{\omega,\epsilon}|)^{\beta}}{\epsilon^{2}}\Big), & \text{in}\ B_r(x_{\omega,\epsilon}), \vspace{0.12cm}\\
                         \overline{C},& \text{in}\ \mathbb{R}^N\backslash B_r(x_{\omega,\epsilon}),
                       \end{array}
                     \right.
\end{equation}
where $\beta >2$ and the constants $\overline{C}>0$, $\nu>0$ will be determined later.

Let
\begin{equation*}
\overline{U}^{\nu,\mu}_{\epsilon}(x) := \epsilon^{2}p^{\nu}_{\epsilon}(x)w_{\mu}(x),
\end{equation*}
where the function $w_{\mu}\in C^2(\mathbb{R}^N)$ with $\inf_{x\in U\backslash\Lambda}w_{\mu}(x)>0$ and $\mu>0$ is defined as
\begin{equation*}
w_{\mu}(x) =\left\{
              \begin{array}{ll}
                \frac{1}{d},\ d = \max_{\Lambda}|x| & \text{if}\ x\in\Lambda, \vspace{0.12cm}\\
                \frac{1}{|x|^{\mu}}, & \text{if}\ x\in\mathbb{R}^N\backslash U.
              \end{array}
            \right.
\end{equation*}
In the sequel, we set
\begin{equation*}
C^1_{\mu,d} := \inf_{U\backslash\Lambda}w_{\mu}(x),\ \ C^2_{\mu,d} := \sup_{U\backslash\Lambda}w_{\mu}(x)\ \ \text{and}\ C^3_{\mu,d} = \sup_{x\in U\backslash\Lambda}|\Delta w_{\mu}(x)|.
\end{equation*}

\vspace{0.2cm}

\noindent\textbf{Case 1}: $\liminf_{|x|\to+\infty}V_{min}(x)|x|^{2\sigma}>0$ with $\sigma\leq 1$.

A direction computation shows that
\begin{align*}
-\epsilon^2\Delta \overline{U}^{\nu,\mu}_{\epsilon}/\overline{C}
  & \geq  \left\{
        \begin{array}{ll}
{\epsilon^2\nu\frac{\beta(N - 1)(r - |x - x_{\omega,\epsilon}|)^{\beta - 1}}{|x - x_{\omega,\epsilon}|}}&\\
\qquad {-\epsilon^2\nu \beta(\beta - 1)(r - |x-x_{\omega,\epsilon}|)^{\beta - 2}}, & x\in B_r(x_{\omega,\epsilon})\backslash B_{R{\epsilon}}(x_{\omega,\epsilon})\vspace{0.12cm},\\
                                        0,& x\in \Lambda\backslash B_r(x_{\omega,\epsilon})\vspace{0.12cm},\\
 -\epsilon^4C^3_{\mu,d},&x\in U\backslash\Lambda,\vspace{0.12cm}\\
      \epsilon^4\frac{(N - 2 - \mu)\mu}{|x|^{\mu + 2}},& x\in
 \mathbb{R}^N\backslash U,
        \end{array}
      \right.
\end{align*}
which implies that
\begin{align*}
&\quad \big(-\epsilon^2\Delta \overline{U}^{\nu,\mu}_{\epsilon} + (1 - \delta)V_{min}(x)\overline{U}^{\nu,\mu}_{\epsilon}\big)/\overline{C}\vspace{0.12cm}\\
  & \geq  \left\{
        \begin{array}{ll}
          \epsilon^2\big({-\nu \beta(\beta - 1)(r - |x-x_{\omega,\epsilon}|)^{\beta - 2}} + (1 - \delta)\textsf{m}_1/d\big), &  x\in B_{r}(x_{\omega,\epsilon})\backslash B_{R{\epsilon}}(x_{\omega,\epsilon}),\vspace{0.12cm}\\
           0, & x\in \Lambda\backslash B_{r}(x_{\omega,\epsilon}),\vspace{0.12cm}\\
         \epsilon^2\textsf{m}_1C^1_{\mu,d} -\epsilon^4C^3_{\mu,d},&x\in U\backslash\Lambda,\vspace{0.12cm}\\
\epsilon^2\frac{1}{|x|^{\mu + 2\sigma}} + \epsilon^{4}\frac{\mu(N - 2 -\mu)}{|x|^{\mu + 2}},&x\in \mathbb{R}^N\backslash U.
        \end{array}
      \right.
\end{align*}

Then, letting $\tilde{\nu} = \nu_{d,\textsf{m}_1,\beta,r}>0$ and $\epsilon_{\mu,d,\textsf{m}_1}>0$ be the two constants such that

\begin{eqnarray}\label{eq5.10}
\begin{split}
&(i)'\ \ -\tilde{\nu}\beta(\beta - 1)r^{\beta - 2} + (1 - \delta)\textsf{m}_1/d\geq 0,\\
&(ii)'\ \ \textsf{m}_1C^1_{\mu,d} - \epsilon^2_{\mu,d,\textsf{m}_1}C^3_{\mu,d} \geq \frac{\textsf{m}_1C^1_{\mu,d}}{2},\\
&(iii)'\ \ \epsilon^2_{\mu,d,\textsf{m}_1}(\mu(N - 2 -\mu))\geq-\frac{1}{2},
\end{split}
\end{eqnarray}
since $\sigma\in[0,1]$, we have
\begin{align*}
\big(-\epsilon^2\Delta \overline{U}^{\tilde{\nu},\tilde{\mu}}_{\epsilon} + (1 - \delta)V_{min}(x)\overline{U}^{\tilde{\nu},\tilde{\mu}}_{\epsilon}\big)/\overline{C}\vspace{0.12cm} \geq  \left\{
        \begin{array}{ll}
           0, & x\in \Lambda,\vspace{0.12cm}\\
           \frac{\epsilon^2\textsf{m}_1C^1_{\mu,d}}{2},&x\in U\backslash\Lambda,\vspace{0.12cm}\\
\frac{\epsilon^2}{2|x|^{\mu + 2\sigma}} ,&x\in \mathbb{R}^N\backslash U,
        \end{array}
      \right.
\end{align*}
if
\begin{equation*}
\epsilon<\epsilon_{\mu,d,\textsf{m}_1}.
\end{equation*}

Define
\begin{align*}\label{eq5.8}
  P_{\epsilon}(x) = \frac{\epsilon^{\delta}}{|x|^{
(2 + \kappa)\sigma}}\chi_{\mathbb{R}^N\backslash\Lambda},
\end{align*}
where $\delta\in(0,4p - 4)$ and $\kappa > 0$ is a small parameter. It is easy to check that \eqref{eq3.1} is satisfied by such $P_{\epsilon}$. Moreover, letting $\overline{C}_{d,\tilde{\nu}}>0$, $\tilde{\mu} = \mu_{p,\sigma}>0\ \text{and}\ \hat{\epsilon}_{\mu_{p,\sigma},\hat{d},U,\textsf{m}_1,\overline{C}_{d,\tilde{\nu}}}>0$ satisfying
\begin{eqnarray*}\label{}
\begin{split}
&(i)''\ \ \inf_{x\in B_{R\epsilon}(x_{\omega,\epsilon})}\overline{C}_{d,\tilde{\nu}}\tilde{\nu}(r - |x - x_{\omega,\epsilon}|)/d\geq \widetilde{C},\\
&(ii)''\ \ (2p - 2)\tilde{\mu} = (2\sigma + 2\kappa),\\
&(iii)''\ \ \frac{\textsf{m}_1C^1_{\tilde{\mu},d}}{2}\geq \frac{\hat{\epsilon}_{\tilde{\mu},\hat{d},U,\textsf{m}_1,\overline{C}_{d,\tilde{\nu}}}^{\delta}C^2_{\tilde{\mu},d}}{\hat{d}^{\tilde{\mu}}},\\
&(iv)''\ \ \hat{\epsilon}_{\tilde{\mu},\hat{d},U,\textsf{m}_1,\overline{C}_{d,\tilde{\nu}}}^{\delta}\leq \frac{1}{2}\\
\text{and}&\ (v)''\ \ {\hat{\epsilon}^{4(p - 1)- \delta}_{\tilde{\mu},\hat{d},U,\textsf{m}_1,\overline{C}_{d,\tilde{\nu}}}\overline{C}^{2p - 2}_{d,\tilde{\nu}}|x|^{(2 + \kappa)\sigma}(w_{\mu}(x))^{2p - 2}}\chi_{\mathbb{R}^N\backslash\Lambda}(x)\leq 1,
\end{split}
\end{eqnarray*}
where $\hat{d} = \min_{U\backslash\Lambda}|x|$, then
there hold
\begin{align*}
  \big(\overline{U}^{\tilde{\nu},\tilde{\mu}}_{\epsilon}(x)\big)^{2p - 2} < P_{\epsilon}(x)\ \ \forall x\in\mathbb{R}^N\backslash\Lambda
\end{align*}
and
\begin{equation*}
  \left\{
    \begin{array}{ll}
     -\epsilon^2\Delta \overline{U}^{\tilde{\nu},\tilde{\mu}}_{\epsilon} + (1 - \delta)V_{min}(x)\overline{U}^{\tilde{\nu},\tilde{\mu}}_{\epsilon}\geq P_{\epsilon}\overline{U}^{\tilde{\nu},\tilde{\mu}}_{\epsilon}, & \text{in}\ \mathbb{R}^N\backslash B_{R\epsilon}(x_{\omega,\epsilon}),\vspace{0.12cm} \\
      \overline{U}^{\tilde{\nu},\tilde{\mu}}_{\epsilon}\geq \widetilde{C}, & \text{in}\  B_{R\epsilon}(x_{\omega,\epsilon}),
    \end{array}
  \right.
\end{equation*}
if
\begin{equation*}\label{}
\left\{
  \begin{array}{ll}
    \overline{C} = \overline{C}_{d,\tilde{\nu}} = \overline{C}_{\textsf{m}_1,d,\nu_{d,\textsf{m}_1,\beta,r}}, &  \\
    \mu = \tilde{\mu}, &\vspace{0.12cm}\\
    \epsilon<\min\{\hat{\epsilon}_{\tilde{\mu},\hat{d},U,\textsf{m}_1,\overline{C}_{d,\tilde{\nu}}},\epsilon_{\mu,d}\} = \min\{\hat{\epsilon}_{\tilde{\mu},d,U,\textsf{m}_1,\overline{C}_{\textsf{m}_1,\hat{d},\nu_{d,\textsf{m}_1,\beta,r}}},\epsilon_{\tilde{\mu},d}\}. &
  \end{array}
\right.
\end{equation*}

As a result, we complete the proof of case 1 by letting
\begin{equation*}
\overline{U}_{\epsilon}(x): = \overline{U_{\epsilon}}^{\tilde{\nu},\tilde{\mu}}(x) := \epsilon^2p^{\tilde{\nu}}_{\epsilon}(x)w_{\tilde{\mu}}(x)\ \ x\in\mathbb{R}^N
\end{equation*}
and
\begin{equation*}
P_{\epsilon}(x) = \frac{\epsilon^{\delta}}{|x|^{(2 + \kappa)\sigma}}\chi_{\mathbb{R}^N\backslash\Lambda}\ \ \text{with}\ \delta\in(0,4p - 4).
\end{equation*}

\vspace{0.2cm}

\noindent\textbf{Case 2:} $N\geq 3$ and $2p>2 + \frac{2}{N - 2}$.

In this case we let $\tilde{\mu} = N - 2$ and
\begin{equation*}\label{eq5.20}
\tilde{w}_{\tilde{\mu}}(x) =\left\{
              \begin{array}{ll}
                1,& \text{if}\ x\in\Lambda, \\
                |x|^{-\tilde{\mu}}(1 - c|x|^{-\varrho}), & \text{if}\ x\in\mathbb{R}^N\backslash U,
              \end{array}
            \right.
\end{equation*}
where $c>0$ is suitably small and
{\begin{equation}\label{**}
\varrho = \frac{(2p - 2)(N - 2) - 2}{4}.
\end{equation}}
It is easy to check that
\begin{equation*}\label{}
  -\Delta \tilde{w}_{\mu} = c\frac{\varrho(\tilde{\mu} + \varrho)}{|x|^{N + \varrho}},\ \ x\in\mathbb{R}^N\backslash U.
\end{equation*}

For simplicity, we denote
$$
\dot{p}:= \frac{3}{2p - 2}.
$$
Define
$$
\overline{U}^{\tilde{\nu},\tilde{\mu}}_{\epsilon}(x) = {\epsilon^{\dot{p}}} p^{\tilde{\nu}}_{\epsilon^{\dot{p}}}(x)\tilde{w}_{\tilde{\mu}}(x),
$$
where $p^{\tilde{\nu}}_{\epsilon^{\dot{p}}}(x)$ is defined in \eqref{eq5.7}. By the same choice of the constants as that in \eqref{eq5.10}, we have
\begin{align*}
\big(-\epsilon^2\Delta \overline{U}^{\tilde{\nu},\tilde{\mu}}_{\epsilon} + (1 - \delta)V_{min}(x)\overline{U}^{\tilde{\nu},\tilde{\mu}}_{\epsilon}\big)/\overline{C}\vspace{0.12cm} \geq  \left\{
        \begin{array}{ll}
           0, & x\in \Lambda,\\
           \frac{(\epsilon^{\dot{p}+2})\textsf{m}_1C^1_{\tilde{\mu},d}}{2},&x\in U\backslash\Lambda,\\
c\frac{\epsilon^{\dot{p} + 2}}{|x|^{N + \varrho}} ,&x\in \mathbb{R}^N\backslash U,
        \end{array}
      \right.
\end{align*}
if $\epsilon>0$ is small enough.

Now we define
\begin{equation*}\label{}
  {P}_{\epsilon}(x) = \frac{\epsilon^{\frac{5}{2}}}{|x|^{2 + 2\varrho}}\chi_{\mathbb{R}^N\backslash\Lambda}.
\end{equation*}
Obviously, \eqref{eq3.1} is satisfied by such $\widetilde{P}_{\epsilon}$. Moreover, letting $\overline{C}\geq 1$ large enough, there hold
\begin{equation*}\label{}
 \big(-\epsilon^2\Delta \overline{U}^{\tilde{\nu},\tilde{\mu}}_{\epsilon} + (1 - \delta)V_{min}(x)\overline{U}^{\tilde{\nu},\tilde{\mu}}_{\epsilon}\big)/\overline{C}\geq P_{\epsilon}(x)\overline{U}^{\tilde{\nu},\tilde{\mu}}_{\epsilon}(x),\ U^{\tilde{\nu},\tilde{\mu}}_{\epsilon}(x)\geq \widetilde{C}\ \ \text{in}\ \ B_{R\epsilon}(x_{w,\epsilon}),
\end{equation*}
and
\begin{equation*}\label{}
 (\overline{U}^{\tilde{\nu},\tilde{\mu}}_{\epsilon})^{2p - 2} = \Big(\frac{\epsilon^{\dot{p}}\overline{C}}{|x|^{\tilde{\mu}}}\Big)^{2p - 2} = \epsilon^3\overline{C}^{2p - 2}|x|^{-(N - 2)(2p - 2)} \leq P_{\epsilon}(x),\ \ \forall x\in\mathbb{R}^N\backslash\Lambda,
\end{equation*}
if $\epsilon>0$ is small enough. This completes the proof.
\end{proof}

\vspace{0.5cm}

\noindent Now we prove that the penalized solution $w_{\epsilon}$ obtained in Lemma \ref{le3.8} solves the original problem.

\begin{proof}[Proof of the existence assertion in Theorem \ref{th1.2}]

Denote $\breve{v}_{\varepsilon} = u^1_{\varepsilon} + u^2_{\varepsilon} - \overline{U}_{\varepsilon}$. Propositions \ref{pr4.1} and \ref{pr5.3} imply that
\begin{equation}\label{eq5.9}
\left\{
  \begin{array}{ll}
    -\epsilon^{2}\Delta \breve{v}_{\epsilon}  - P_{\epsilon}\breve{v}_{\epsilon} + ((1 - \delta)V)\breve{v}_{\epsilon} \leq 0, & \text{in}\ \mathbb{R}^N\backslash B_{R{\epsilon}}(x_{\omega,\epsilon}),\vspace{2mm}\\
    \breve{v}_{\epsilon}\leq 0, & \text{on}\ B_{R{\epsilon}}(x_{\omega,\epsilon}).
  \end{array}
\right.
\end{equation}

\vspace{0.2cm}

\noindent\textbf{Case 1:} $\liminf_{|x|\to+\infty}V_{min}(x)|x|^{2\sigma}>0$ with $\sigma\in[0,1]$.

In this case, the penalized function $P_{\epsilon}$ satisfies
\begin{equation*}\label{}
  P_{\epsilon}(x)\leq V(x),\ \ \forall x\in\mathbb{R}^N.
\end{equation*}
Then by comparison principle we have $\breve{v}_{\epsilon}\leq 0$ in $\mathbb{R}^N$. Hence $(u^1_{\epsilon} + u^2_{\epsilon})^{2p - 2}\leq \big(\overline{U}^{\tilde{\nu},\tilde{\mu}}_{\epsilon}\big)^{2p - 2}$ in $\mathbb{R}^N\backslash B_{R\epsilon}(x_{\omega,\epsilon})$. Furthermore, $ (u^1_{\epsilon} + u^2_{\epsilon})^{2p - 2} < P_{\epsilon}$ on $\mathbb{R}^N\backslash\Lambda$.

\vspace{0.2cm}

\noindent\textbf{Case 2:} $N\geq 3$ and $2p - 2>\frac{2}{N - 2}$.

In this case, we have 
\begin{equation*}\label{}
  \left\{
  \begin{array}{ll}
    \Big(-\Delta  - P_{\epsilon}(x)\epsilon^{-2}\Big)\breve{v}_{\epsilon}\leq 0, & \text{in}\ \mathbb{R}^N\backslash B_{R{\epsilon}}(x_{\omega,\epsilon}),\vspace{2mm}\\
    \breve{v}_{\epsilon}\leq 0, & \text{on}\ B_{R{\epsilon}}(x_{\omega,\epsilon}).
  \end{array}
\right.
\end{equation*}
But, by Hardy's inequality \eqref{CCeq3.3} we know that the operator $-\Delta - \epsilon^{-2}P_{\epsilon}(x)$ is a positive operator. Hence we also conclude that $(u^1_{\epsilon} + u^2_{\epsilon})^{2p - 2} < P_{\epsilon}$ on $\mathbb{R}^N\backslash\Lambda$.

As a result, the penalized solution $w_{\epsilon}$  constructed in Lemma \ref{le3.8} is the solution of the original problem \eqref{eq1.1}.
\end{proof}

\section{Verifying the assumptions in Lemma \ref{LE4.6}}\label{s6}

In this section, we are going to verify the decay estimates in Theorem \ref{th1.2}, which imply the assumptions in Lemma \ref{LE4.6}.
We split the argument into two subsections with respect to the two conditions \eqref{AAeq1.8} and \eqref{AAeq1.9}.

\subsection{The  slow decay case: $0\leq \sigma\leq1$}
The penalized function $P_{\epsilon}$ constructed in Proposition \ref{pr5.3} means that
\begin{equation*}\label{}
 (v_{\epsilon}(x))^{2p - 2}:= (u^1_{\epsilon}(x) + u^2_{\epsilon}(x))^{2p - 2}\leq \frac{\epsilon^{\delta}}{|x|^{(2 + \kappa)\sigma}},
\end{equation*}
from which and \eqref{eq5.2} we deduce that the function $\tilde{v}_{\epsilon}(x) = v_{\epsilon}(\epsilon x + x_{\omega,\epsilon})$ satisfies
\begin{equation}\label{eq6.2}
\left\{
  \begin{array}{ll}
    -\Delta \tilde{v}_{\epsilon} + (1 - \delta/2)\widetilde{V}_{\epsilon}(x)\tilde{v}_{\epsilon}(x)\leq 0, & x\in\mathbb{R}^N\backslash B_R, \\
    \tilde{v}_{\epsilon}(x)\leq \widetilde{C}, & x\in B_R,\\
    \tilde{v}_{\epsilon}(x)\to 0\ \ \text{as}\ |x|\to \infty,
  \end{array}
\right.
\end{equation}
where $\widetilde{V}_{\epsilon}(x) = V(\epsilon x + x_{\omega,\epsilon})$. 

It was proved in \cite{AM-061} that for every $m > 0$, there exists $\tilde{R}>0$ and $\epsilon_0 > 0$ such that
\begin{equation*}\label{}
  \widetilde{V}_{\epsilon}(x)\ge \frac{m}{|x|^{2\sigma}},\ \text{for all}\ |x|\ge\tilde{R}\ \text{if}\ 0<\epsilon<\epsilon_0,
\end{equation*}
by which and \eqref{eq6.2}, we conclude without loss of generality that
\begin{equation*}\label{eq6.4}
\left\{
  \begin{array}{ll}
    -\Delta \tilde{v}_{\epsilon} + \frac{m}{|x|^{2\sigma}}\tilde{v}_{\epsilon}(x)\leq 0, & x\in\mathbb{R}^N\backslash B_{\tilde{R}},\\
    \tilde{v}_{\epsilon}(x)\leq \widetilde{C}, & x\in B_{\tilde{R}},\\
    \tilde{v}_{\epsilon}(x)\to 0\ \ \text{as}\ |x|\to \infty.
  \end{array}
\right.
\end{equation*}
Then by the results in \cite{AM-061} again, we have the following Proposition.
\begin{proposition}
Let $\tilde{v}_{\epsilon}$ satisfy \eqref{eq6.2}. Then for every $|x|>\tilde{R}$, it holds
\begin{equation*}\label{}
  \tilde{v}_{\epsilon}(x) \leq \left\{
\begin{array}{ll}
      \widetilde{C}^1_{\tilde{R}}e^{-\frac{\sqrt{m}}{1 - \sigma}|x|^{1 - \sigma}}, & \text{if}\ 0<\sigma<1, \\
      \widetilde{C}^2_{\tilde{R}}|x|^{\frac{2 - N - \sqrt{(N - 2)^2 + 4m}}{2}}& \text{if}\ \sigma = 1,
   \end{array}
 \right.
\end{equation*}
where $\widetilde{C}^i_{\tilde{R}}$, $i = 1,2$, are suitably positive constants.
\end{proposition}
As  a result, returning to $v_{\epsilon}$, we have for every $x\in\R^N$
\begin{equation*}\label{}
  {v}_{\epsilon}(x) \leq \left\{
\begin{array}{ll}
      \widetilde{C}^1_{\tilde{R}}e^{-\frac{\sqrt{m}}{1 - \sigma}\epsilon^{\sigma - 1}|x - x_{\omega,\epsilon}|^{1 - \sigma}}, & \text{if}\ 0<\sigma<1, \\
      \widetilde{C}^2_{\tilde{R}}\epsilon^{\frac{N - 2 + \sqrt{(N - 2)^2 + 4m}}{2}} |x - x_{\omega,\epsilon}|^{\frac{2 - N - \sqrt{(N - 2)^2 + 4m}}{2}},& \text{if}\ \sigma = 1.
   \end{array}
 \right.
\end{equation*}
Moreover, since
\begin{equation*}\label{}
  -\Delta v = \frac{1}{\epsilon^2}\Big(\sum_{i = 1}^2\big((u^i_{\epsilon})^{2p - 1} - V_i(x)u^i_{\epsilon}\big) +  \beta(u^1_{\epsilon})^{p - 1}(u^2_{\epsilon})^{p - 1}v_{\epsilon}\Big): = \frac{1}{\epsilon^2}f_{\epsilon},
\end{equation*}
by $L^q$-estimate in \cite{GT}, for every $z\in\partial B_{\delta}(x_{\omega,\epsilon})$, letting $q > N$ and $m$ big enough when $\sigma = 1$, it holds
\begin{eqnarray*}\label{eq6.8}
\begin{split}
\|v_{\epsilon}\|_{W^{2,q}(B_{\frac{\delta}{4}}(z))}
&\leq \frac{C}{\epsilon^2}\|f_{\epsilon}\|_{L^q(B_{\frac{\delta}{2}}(z))} + C\|v_{\epsilon}\|_{L^q(B_{\frac{\delta}{2}}(z))}\\
&\leq \frac{C}{\epsilon^2}\|v_{\epsilon}\|_{L^q(B_{\frac{\delta}{2}}(z))}\\
&\leq
\left\{
       \begin{array}{ll}
         Ce^{-c/\epsilon^{1-\delta}}, & \text{if}\ 0\leq\sigma < 1,\\
         C\epsilon^{\frac{N - 2 + \sqrt{(N - 2)^2 + 4m}}{2} - 2}, & \text{if}\ \sigma = 1,
       \end{array}
     \right.
\\
&=\epsilon^No_{\epsilon}(1),
\end{split}
\end{eqnarray*}
which and Sobolev embedding imply that
$$
|\nabla v_{\epsilon}(z)|^2  = \epsilon^No_{\epsilon}(1)\ \text{for all}\ z\in\partial B_{\delta}(x_{\omega,\epsilon}).
$$
Note that the estimate also holds for single $u^i_{\epsilon}$, $i = 1,2$. This gives the assumptions in Lemma \ref{LE4.6}.

\subsection{Faster decay or compactly supported case}
The decay estimates in this case indeed can be obtained by the argument in \cite{MV}. But because of our penalized function in Section \ref{s5} enjoys better decay rates, the proof can be more intuitive than that in \cite{MV}.

Our penalized function here is
$$
\frac{\epsilon^{5/2}}{|x|^{2 + 2\varrho}}\chi_{\mathbb{R}^N\backslash\Lambda},
$$
where $\varrho$ is given in \eqref{**}. Easily, replacing the penalized functions in \cite{MV} with $P_{\epsilon}$, we can prove by the same argument as that of Lemma 3.4 in \cite{MV} that the equation
\begin{equation}\label{Feq6.9}
  \left\{
    \begin{array}{ll}
      -\Delta u - \frac{\epsilon^{1/2}}{|x|^{2 + 2\varrho}}u = 0,& x\in\Lambda^c,\\
      u = 1,&x\in\partial\Lambda,
    \end{array}
  \right.
\end{equation}
has a solution $u$ if $\epsilon>0$ is small. Moreover, there exist constants $c,C>0$ such that
\begin{equation}\label{Feq6.10}
 \frac{c}{|x|^{N - 2}}\le u(x)\le \frac{C}{|x|^{N - 2}}.
\end{equation}

According to \eqref{Feq6.9} and \eqref{Feq6.10}, we let $\tilde{u}$ be the positive extension of ${u}{\cosh\textsf{m}_1(\frac{r}{\epsilon} - R)}$ satisfying $\tilde{u}_{\epsilon}(x) = {\cosh\textsf{m}_1(\frac{r}{\epsilon} - R)}$ if $d(x,\Lambda^c)\ge r$ and set
\begin{equation*}\label{}
  \hat{v}_{\epsilon}(x) =\left\{
                             \begin{array}{ll}
                               \cosh\frac{\textsf{m}_1(r - |x - x_{\omega,\epsilon}|)}{\epsilon}, &\text{if}\ x\in B_{r}(x_{\omega,\epsilon}), \\
                            \tilde{u}_{\epsilon}(x),   &\text{if}\ x\in B^c_r(x_{\omega,\epsilon}),
                             \end{array}
                           \right.
\end{equation*}
where
$$
\cosh t = \frac{e^t + e^{-t}}{2}.
$$
Then since $B_{2r}(x_{\omega,\epsilon})\subset\subset \Lambda$ if $\epsilon$ is small enough(see \eqref{space}), we have $\hat{v}_{\epsilon}\in C^2(\mathbb{R}^N)$. Moreover, by the same computation as that in Lemma 5.2 of \cite{MV}, we conclude that
\begin{equation}\label{Feq6.12}
  -\epsilon^2 \Delta \hat{v}_{\epsilon} - P_{\epsilon}\hat{v}_{\epsilon} + (1 - \delta)V_{min}\hat{v}_{\epsilon}\ge 0\ \text{in}\ \R^N\backslash B_{\epsilon R}(x_{\omega,\epsilon}).
\end{equation}
Then, letting
$$
\widehat{U}_{\epsilon}(x) = \frac{\hat{v}_{\epsilon}(x)}{\cosh\textsf{m}_1\Big(\frac{r}{\epsilon} - R\Big)},
$$
by \eqref{Feq6.12}, \eqref{eq5.5} and comparison principle in Lemma 3.2 in \cite{MV}, we have
$$
v_{\epsilon}(x)\le \widehat{C}\widehat{U}_{\epsilon}(x)\le Ce^{-\frac{\textsf{m}_1|x - x_{\omega,\epsilon}|}{\epsilon(|1 + |x - x_{\omega,\epsilon}|)}}\frac{1}{1 + |x|^{N - 2}}.
$$

Finally, for a $0<\delta<r$, it holds
\begin{eqnarray*}\label{}
\begin{split}
  v_{\epsilon}(x)
&\leq Ce^{-\frac{\textsf{m}_1|x - x_{\omega,\epsilon}|}{\epsilon(|1 + |x - x_{\omega,\epsilon}|)}} \ \text{for all}\ x\in B_{\delta}(x_{\omega,\epsilon}),
\end{split}
\end{eqnarray*}
by which and the same argument above, we get the assumptions in Lemma \ref{LE4.6}.

\end{document}